\documentclass[11pt]{amsart}
\usepackage{amssymb,mathrsfs,graphicx,enumerate,color}
\usepackage{epsfig}
\usepackage{caption}
\graphicspath{{Figure/}}
\usepackage[margin=1in]{geometry}
\usepackage{subcaption}
\usepackage{hyperref}
\usepackage{comment}
\usepackage{colortbl}
\definecolor{black}{rgb}{0.0, 0.0, 0.0}
\definecolor{red}{rgb}{1.0, 0.5, 0.5}

\title[Weak-BV Stability of the Euler System]{Weak-BV stability of inflow and outflow problems for the isentropic Euler system}

\author[Kang]{Moon-Jin Kang}
\address[Moon-Jin Kang]{\newline Department of Mathematical Sciences \newline Korea Advanced Institute of Science and Technology, Daejeon  34141, Republic of Korea}
\email{moonjinkang@kaist.ac.kr}

\author[Meng]{Jiayun Meng}
\address[Jiayun Meng]{\newline Department of Mathematics \newline The University of Texas at Austin, Austin, TX 78712, USA}
\email{jiayun@utexas.edu}

\author[Oh]{HyeonSeop Oh}
\address[HyeonSeop Oh]{\newline Department of Mathematical Sciences \newline Korea Advanced Institute of Science and Technology, Daejeon  34141, Republic of Korea}
\email{ohs2509@kaist.ac.kr}

\author[Vasseur]{Alexis F. Vasseur}
\address[Alexis F. Vasseur]{\newline Department of Mathematics \newline The University of Texas at Austin, Austin, TX 78712, USA}
\email{vasseur@math.utexas.edu}

\newtheorem{theorem}{Theorem}[section]

\newtheorem{conv}{Convention}[section]

\newcommand{\e}{\varepsilon}

\numberwithin{figure}{section}
\newcommand{\beq}{\begin{equation}}
\newcommand{\eeq}{\end{equation}}
\newcommand{\bsp}{\begin{split}}
\newcommand{\esp}{\end{split}}

\newcommand{\sgn}{{\text{\rm sgn}}}

\newcommand{\RR}{{\mathbb R}}

\newcommand{\s}{\sigma}

\newcommand\adots{\mathinner{\mkern2mu\raise1pt\hbox{.}
\mkern3mu\raise4pt\hbox{.}\mkern1mu\raise7pt\hbox{.}}}

\newcommand{\lct}{C_b}
\newcommand{\ts}{t_*}

\newtheorem{theo}{Theorem}[section]
\newtheorem{prop}[theo]{Proposition}

\newtheorem{lem}[theo]{Lemma}
\newtheorem{defi}[theo]{Definition}

\newtheorem{rem}[theo]{Remark}

\def\charf {\mbox{{\text 1}\kern-.30em {\text l}}}

\newcommand{\util}{{\tilde{u}}}

\def \m {\mu}

\def \r {\rho}
\def \rhotil {\tilde{\rho}}
\def \b {\beta}

\def \Rp {\mathbb{R}_+}

\def \x {\xi}
\def \th {\theta}

\def \Util {\widetilde{U}}

\def \Ubar {\overline{U}}

\newcommand{\del}{\Delta}

\begin{document}
\bibliographystyle{acm}

\date{\today}

\subjclass[2020]{35L65, 35Q35, 35L50, 35B35} \keywords{Compressible Euler System; Inflow problem; Outflow problem; Stability.}

\thanks{\textbf{Acknowledgment.}  M.-J. Kang and H. Oh were partially supported by the National Research Foundation of Korea (RS-2024-00361663). A. Vasseur was partially supported by the NSF grants DMS 2219434, 2306852, and 1840314.} 

\begin{abstract}
We study the well-posedness of small BV solutions to the one-dimensional isentropic Euler system on the half-line under inflow and outflow boundary conditions. These boundary conditions are formulated in terms of admissible trace sets determined by Navier--Stokes boundary layers and zero-speed shocks. For both the inflow and outflow problems, we construct small BV solutions taking values in the subsonic region. These solutions are unique and satisfy the quantitative stability estimate
\begin{align*}
    \|U(\cdot, t)-V(\cdot, t)\|_{L^2} \lesssim \sqrt{\|U(\cdot, 0) - V(\cdot, 0)\|_{L^2}} \, ,
\end{align*}
which holds for any such small BV solution $V$ and any $L^\infty$ entropy solution $U$ satisfying the strong trace property and the inflow/outflow boundary conditions. In particular, we emphasize that $U$ may take values outside the subsonic region and need not have bounded variation or be constructed by the front tracking scheme. 
This establishes the first weak-BV stability and uniqueness theory for inflow and outflow initial-boundary value problems. A key difficulty lies in verifying that the boundary set is compatible with the $a$-contraction framework, and in proving that the front-tracking limit satisfies the prescribed inflow/outflow conditions.
\end{abstract}

\maketitle \centerline{\date}

\tableofcontents
\section{Introduction}
\setcounter{equation}{0}

We consider the isentropic Euler system on the half space $\mathbb{R}_+ := (0,\infty)$: 
\begin{equation}\label{eq:E}
	\begin{aligned}
		&\r_t + (\r u)_x = 0, \quad x > 0, \quad t > 0,\\
		&(\r u)_t + (\r u^2 + p(\r))_x = 0, \quad x > 0, \quad t > 0,
	\end{aligned}
\end{equation}
where $\r(x,t), u(x,t)$, and $p(\r)$ represent the fluid density, fluid velocity, and pressure law. We consider the case of a polytropic perfect gas, for which the pressure satisfies
\[
p(\r) = \r^\gamma, \quad \gamma >1,
\]
where $\gamma$ denotes the adiabatic constant. We denote by $\Omega := \{U=(\r,\r u):\r \geq 0\}$ the state space of \eqref{eq:E}.
Throughout this paper, we work with the bounded state space
\begin{align}\label{def:RI}
    \mathcal{V}_0 = \{U=(\rho, \rho u)\in { \overline{\mathbb{R}}_+}\times\mathbb{R}: -C < V_1(U) = u-c_1\rho^{\frac{\gamma-1}{2}} \leq u + c_1\rho^{\frac{\gamma-1}{2}} = V_2(U) < C\} \subset \Omega ,
\end{align}
where $c_1:=\frac{2\sqrt{\gamma}}{\gamma -1}$ and $V_i(U)$ is the $i$-th Riemann invariant of the system \eqref{eq:E}. We define $\mathcal{V}$ as the interior of $\mathcal{V}_0$. Note that $\mathcal{V}_0$ includes the vacuum state $\{(0,0)\}$, and $\mathcal{V}_0 = \mathcal{V} \cup \{(0,0)\}$.\\

We will only consider entropic solution of (\ref{eq:E}), i.e., the weak solution that verifies 
\begin{align}\label{entropy inequality}
    \partial_t \eta(U) + \partial_x q(U) \leq 0 \, , \; \textup{ for any } t>0 \, ,\; x>0 \, .
\end{align}
in the sense of distributions, where $q$ is the entropy flux associated with the entropy $\eta$. More precisely, we consider solutions that satisfy that for any $0\leq \psi \in C_0^\infty([0,\infty)\times [0,\infty))$,
\begin{align*}
    \int_0^\infty \int_0^\infty \eta(U)\partial_t \psi + q(U)\partial_x \psi \, dx\, dt + \int_0^\infty \eta(U^0) \psi(x,0) \, dx + \int_0^\infty q(U(0,t)) \psi(0,t) \, dt \geq 0 \, ,
\end{align*}
where $U^0: \mathbb{R}_+ \rightarrow \mathbb{R}$ is the prescribed initial data.\\

\noindent {\bf$\bullet$ Boundary conditions for the Euler system.} 
The eigenvalues of the flux Jacobian in (\ref{eq:E}) are
\begin{align*}
    \lambda_1(U) = u - \sqrt{\gamma\,\rho^{\gamma-1}},
    \qquad
    \lambda_2(U) = u + \sqrt{\gamma\,\rho^{\gamma-1}}.
\end{align*}
Note $\lambda_1(U)<\lambda_2(U)$ for all states $U$.

In the state space, we distinguish five regions according to the sign of the characteristic speeds:
\begin{equation*}
    \begin{aligned}
        \Omega_{super}^+ = \{U: \lambda_1(U)>0\} \, , \; \Omega_{sub} = \{U: 
        \lambda_1(U)<&\,0<\lambda_2(U)\} \, , \; \Omega_{super}^- = \{U: \lambda_2(U)<0\} \, ,  \\
        \Gamma_{trans}^+ = \{U: \lambda_1(U)=0\} \, &, \; \Gamma_{trans}^- = \{U: \lambda_2(U)=0\} \, . \\
    \end{aligned}
\end{equation*}
Here, $\Omega_{sub}$ denotes the subsonic region, $\Omega_{super}^\pm$ denote the supersonic regions, and $\Gamma_{trans}^\pm$ denote the transonic curves.

The dimension of boundary conditions for the Euler system depends on the above regions.
We study two boundary value problems for the isentropic Euler system on the half space $\mathbb{R}_+$: the inflow and outflow problems on subsonic regions. 
In this paper, we first construct small BV solution satisfying the boundary condition \eqref{small BV BC} as a one-dimensional boundary set near a base point $U_b$ on the subsonic region. We also prove the stability of such a small BV solutions under perturbations that are bounded entropy solution satisfying strong trace property and boundary condition. We consider a large boundary set for perturbations, which can even cross the transonic curve into supersonic region. We define this set based on the boundary conditions introduced by Ancona, Marson, and Spinolo in \cite{AMS24} (see also \cite{BS20}). 
This boundary set contains points connected by boundary layer of Navier-Stokes system or of 0-speed shock starting from the base point. To introduce such a large boundary set, we consider the inflow and outflow problem of the Navier-Stokes system
\begin{equation}\label{eq:NS}
	\begin{aligned}
		&\r_t + (\r u)_x = 0, \quad x\geq 0, \quad t \geq 0,\\
		&(\r u)_t + (\r u^2 + p(\r))_x = \m u_{xx},
	\end{aligned}
\end{equation}
subject to the boundary data
\begin{align}\label{NS BC}
    \r(0,t) = \r_b\, , \; u(0,t) = u_b > 0 \, \textup{ for the inflow and } u(0,t) = u_b < 0 \, \textup{ for the outflow.}
\end{align}
Here, $\m>0$ denotes the vanishing viscosity parmeter. \\%

Formally, the isentropic Euler system \eqref{eq:E} can be regarded as the inviscid limit of the isentropic Navier-Stokes system. In contrast to the Cauchy problem, due to the rapid change of the solution to \eqref{eq:NS} in the vicinity of the boundary as $\mu \to 0+$, the boundary condition \eqref{NS BC} imposed on the Navier-Stokes system is incompatible with the Euler system \eqref{eq:E}, This mismatch is compensated by the formation of a boundary layer.

Now, we present the definition of the boundary conditions for the Euler system \eqref{eq:E}.
\begin{defi}\label{def:BVP1}\cite{AMS24}
For a fixed $U_b=(\rho_b,\rho_b u_b)$, suppose that $u_b>0$ ($u_b<0$, respectively). We say that $\Ubar \sim U_b$ if there exists $\underbar{U}$ such that the following conditions are both satisfied:
\begin{enumerate}
    \item[(i)] $\underline{U}$ (left) and $\Ubar$ (right) are connected by a $0$-speed shock;
    \item[(ii)] there is a ``boundary layer'' $U(x):\Rp \to \Omega$ such that
    \[
    \begin{aligned}
        &(\r u)_x = 0, \quad x\geq 0,\\
		&(\r u^2 + p(\r))_x = \mu u_{xx}\\
        &\beta(U(0),U_b)=0, \quad \lim_{x \to +\infty} U(x) = \underline{U},
    \end{aligned}
    \]
    where $\beta(U,U_b)$ is defined as 
    \[
    \begin{cases}
        \beta(U,U_b) = u - u_b \, \, &\text{if } u_b <0 \quad (outflow),\\
        \beta(U,U_b) = U - U_b \, \, &\text{if } u_b >0 \quad (inflow).
    \end{cases}
    \]
\end{enumerate}
\end{defi}
\begin{defi}\label{def:BVP}
    For a fixed $U_b = (\r_b,\r_bu_b)$, suppose that $u_b>0$ ($u_b<0$ respectively). We define $\mathcal{B}_{in}(U_b)$ ($\mathcal{B}_{out}(U_b)$, respectively) as the boundary set for the inflow (outflow, respectively) problem of the Euler system by
\[
\mathcal{B}_{in/out}(U_b):=\{\Ubar \in \Omega: \Ubar \sim U_b\}.
\]
\end{defi}
Roughly speaking, $\Ubar \in \mathcal{B}_{in/out}(U_b)$ if $U_b$ and $\Ubar$ are connected by a 0-speed shock and by a boundary layer. An detailed analysis of the boundary sets $\mathcal{B}_{in/out}(U_b)$ is presented in Section~\ref{sec:bdset}.\\

Given the boundary datum $U_b$, we say that an entropy solution $U(x,t)$ to \eqref{eq:E} satisfies the inflow/outflow boundary condition if $U(0,\cdot) \sim U_b$, or equivalently,
\[
U(0,\cdot) \in \mathcal{B}_{in/out}(U_b).
\]
\begin{rem}
1. For a fixed $U_b$, the admissible trace set $\mathcal{B}_{out/in}(U_b)$ in Definition~\ref{def:BVP} depends on the underlying viscous mechanism. In fact, if we consider a viscous system different from the Navier-Stokes system, then the sets $\mathcal{B}_{out/in}(U_b)$ generally change.

2. Another natural way to impose the boundary condition, when the boundary is characteristic, is the condition (C) introduced by Amadori \cite{A97}, which generalizes the boundary condition of Dubois, Lefloch \cite{DL88}. Condition (C) is formulated so that, for a given $U_b \in \Omega$, $U_b$ (left) and the trace of a solution to Euler system $U(t,0)$ are connected by shocks or rarefactions with nonpositive speed. We highlight that the admissible boundary sets arising from condition (C) and from Definition~\ref{def:BVP1} (or \cite[Definition 1.1]{AMS24}) are different (See \cite[Section 2.6]{DL88} and Section 2.1). Specifically, the sets $\mathcal{B}_{out/in}(U_b)$ in Definition~\ref{def:BVP} (or \cite[Definition 1.1]{AMS24}) depends on the viscous effects of Navier-Stokes system \eqref{eq:NS}, while the condition (C) is independent of the specific viscous mechanism.
\end{rem}

We only discuss the Navier-Stokes system to determine the Euler boundary set for weak solutions. Throughout the rest of the paper, all solutions are understood to refer to the Euler system.\\

\noindent {\bf $\bullet$ Weak-BV framework: notions of the solution classes $\mathcal{S}_{BV,\e}$ and $\mathcal{S}_{weak}$.} 
We require that the boundary trace of the weak solutions belong to the boundary set, in the sense of the Strong Trace Property 
defined below.
\begin{defi}\label{def stp}
    Let $U\in L^\infty(\mathbb{R}_+ \times \mathbb{R}_+)$. We say that $U$ verifies the Strong Trace Property if for any Lipschitzian curve $t\mapsto X(t) \in [0,\infty)$, there exists two bounded functions $U_-, U_+ \in L^\infty (\mathbb{R}_+)$ such that for any $T>0$,
    \begin{align*}
        \underset{n\rightarrow \infty}{\textup{lim}} \int_0^T \underset{\{y\in(0,1/n) : \, X(t)\pm y \in \mathbb{R}_+\}}{\textup{sup}} |U(t,X(t)\pm y)-U_\pm(t)| \, dt  = 0 \, .
    \end{align*}
\end{defi}
More precisely, let $U_+^b(t)$ be the trace of $U$ on the boundary, i.e. $U_+(t)$ corresponding to $X(t) \equiv 0$. The broader class of weak solutions considered for the inflow/outflow problem is the following.
\begin{equation*}
    \begin{aligned}
        \mathcal{S}_{weak} = \{&U\in L^\infty(\mathbb{R}_+ \times \mathbb{R}_+ :\mathcal{V}_0) :\, U \textup{ is a solution to }  (\ref{eq:E}) \, , \\
        &\textup{satisfying } (\ref{entropy inequality}) \, , \; \textup{Definition}~\ref{def stp}\, , 
        \textup{ and } U_+^b(t) \in \mathcal{B}_{in/out}(U_b) \, \, \text{for a.e. } t > 0 \} \, .
    \end{aligned}
\end{equation*}
Note that the space $\mathcal{S}_{\mathrm{weak}}$ includes vacuum states and has
no smallness condition. There also exist functions with infinite BV norm that
satisfy the Strong Trace Property (see, for example, Section 1.3 in \cite{CKF25}).

On the other hand, for any open and convex set $\mathcal{O} \subset \mathcal{V}$ and any $U_b \in \mathcal{O}$, we define the following set as the class of admissible initial data for small BV solutions:
\begin{equation}\label{small BV initial set}
    \begin{aligned}
        \mathcal{S}_{BV,\varepsilon}^0 = \{U^0: \mathbb{R}_+\rightarrow \mathcal{O} \cap \Omega_{sub} : \|U^0-U_b\|_{L^\infty(\mathbb{R}_+)} \leq \varepsilon \, , \; \|U^0\|_{BV(\mathbb{R}_+)} \leq \varepsilon
        \} \, .
    \end{aligned}
\end{equation}

\subsection{Main results}
We prove the existence and stability of small BV solutions with initial data in
$\mathcal{S}_{BV,\varepsilon}^0$ within the weak solution class
$\mathcal{S}_{\mathrm{weak}}$. Before stating the main results, we first distinguish the inflow and outflow portions of the subsonic region by
\begin{align*}
    \Omega_{sub}^{in} = \Omega_{sub}\cap \{u>0\} \, , \; \Omega_{sub}^{out} = \Omega_{sub}\cap \{u<0\} \, .
\end{align*}
\begin{theorem}\label{thm}
    Let $U_b \in \Omega_{sub}^{in/out}$.
    For any open set $\mathcal{O}$ such that $U_b \in \overline{\mathcal{O}}\subset \mathcal{V} \cap \Omega_{sub}$, there exists $\varepsilon>0$ such that the following is true. 
    
    Let $V^0\in \mathcal{S}_{BV,\varepsilon}^0$. Then there exists a unique small BV solution $V$ to \eqref{eq:E} satisfying \eqref{entropy inequality}, $V(0,t) \in \mathcal{B}_{in/out}(U_b)$, and $V(x,t) \in \Omega_{sub}^{in/out}$ for a.e. $t>0$, 
    with initial data $V^0$.
    Moreover, for any $R,T>0$, there exists a constant $C>0$ depending on $R,T$ such that for any $\tau \in [0,T]$ and any $U\in \mathcal{S}_{weak}$,
    \begin{align}\label{thm eq}
        \|U(\cdot, \tau)-V(\cdot, \tau)\|_{L^2(0,R)} \leq C \sqrt{\|U(\cdot,0)-V(\cdot,0)\|_{L^2(0,R+s\tau)}} \, ,
    \end{align}
    where $s>0$ depending only on $\mathcal{V}_0$ and $\mathcal{O}$, and is precisely defined in (\ref{def of s}).
\end{theorem}

\begin{rem}
    1. Theorem~\ref{thm} establishes existence and uniqueness of small BV solutions in the subsonic region, and their stability within the class $\mathcal{S}_{\mathrm{weak}}$ via the quantitative estimate \eqref{thm eq}. These results extend the result of the Cauchy problem \cite{CKF25} to inflow/outflow boundary value problems. It is worth emphasizing that, from the stability viewpoint, establishing the stability of even a single reflected shock at the boundary has been a highly challenging problem. The present paper establishes the stability of general small BV entropy solutions.
    The small BV solution is explicitly constructed as the limit of front-tracking approximate solutions. 
   
    2. To preserve the non-characteristic inflow/outflow structure for small BV solutions, we restrict their state space to the subsonic region. However, we emphasize that a perturbation $U \in \mathcal{S}_{weak}$ may belong to transonic and supersonic regions. There is no smallness assumption for perturbations belonging to the set $\mathcal{S}_{\mathrm{weak}}$, and it does not arise from the front-tracking algorithm.

    3. The impermeable boundary value problem is proved in \cite{Impermeable} by a shortcut argument, extending the solution from the half-space to the whole space.
\end{rem}

\subsection{Literature review}
In the setting of the Cauchy problem, Glimm \cite{Glimm} proved the existence of a small BV solution $U$ for any initial data $U^0$ defined on the whole line and satisfying the smallness conditions analogous to (\ref{small BV initial set}), provided that $\varepsilon$ is sufficiently small. More precisely, there exists a uniform $C>0$ such that for any $t\geq 0$,
\begin{align*}
    \|U(\cdot, t)\|_{BV(\mathbb{R})} \leq C \varepsilon \, .
\end{align*}
The notion of a small BV solution captures the global existence and the uniform boundedness of the total variation for all time, provided that the initial total variation is sufficiently small and close to $U_b$.
Aside from Glimm's method, often referred to as the Glimm scheme or the random choice method, two other general frameworks are known to yield small BV existence for systems of hyperbolic 
conservation laws: the front tracking scheme \cite{BJ98, B92, D76} and the vanishing viscosity method \cite{vvm}. 

Using the $L^1$ front-tracking framework, the uniqueness of small BV solutions for general $n\times n$ hyperbolic systems was established under several structural assumptions, including the Tame Oscillation Condition \cite{Tame1,Tame2}, the Bounded Variation Condition \cite{BVC}, and the Liu entropy condition \cite{Liu}. These results relied on the Standard Riemann Semigroup constructed in
\cite{BC95,BLY99}. More recently, these technical assumptions were removed in \cite{BGu,Liu}. Nevertheless, uniqueness results based on front tracking apply only within the class of weak solutions with small bounded variation. To establish a more general Weak/Strong uniqueness principle, Chen, Krupa, and Vasseur \cite{qualitative} employed an $L^2$-based approach involving, in particular, $a$-contraction with shifts. For $2\times2$ hyperbolic systems endowed with a strictly convex entropy, they proved the stability of small BV solutions within the broader class of bounded entropy weak solutions satisfying the Strong Trace Property.

We now turn to the initial-boundary value problem. For scalar conservation laws, Bardos, le Roux and N\'ed\'elec \cite{BRN79} extended the Kru\v{z}kov theory \cite{K70} to the initial-boundary value problem. Goodman \cite{G83} established the existence of the small BV solutions for systems of hyperbolic conservation laws via the Glimm-scheme, assuming a non-characteristic boundary. Subsequently, Amadori~\cite{A97} proved the existence of small BV solutions by employing a front-tracking scheme, in a framework that also covers characteristic boundaries and allows for general (possibly time-dependent) boundary profiles.
The work of Amadori and Colombo \cite{AC97} provides an $L^1$ well-posedness theory for $2 \times 2$ systems, where uniqueness and stability are ensured within the semigroup they construct. This analysis was later extended by Donadello and Marson \cite{DM07} to $N \times N$ systems with non-characteristic boundaries. Compared to earlier works, Ancona, Marson and Spinolo \cite{AMS24} proposed a different approach to assigning boundary conditions for $N \times N$ hyperbolic conservation laws. Their boundary condition is derived from the corresponding viscous system, and they established the existence of small BV solutions via front tracking scheme. We also refer Gisclon \cite{G96} and Rousset \cite{R05} for the study of the inviscid limit for regular data in finite time.

In this paper, we extend the existence and uniqueness theory of small BV solutions for general small BV initial data, namely the class $\mathcal{S}_{BV,\varepsilon}^0$ defined in \eqref{small BV initial set}, which was first established in the whole space in \cite{CKF25}, to the initial-boundary value problem. We obtain uniqueness among all weak solutions that satisfy the strong trace property, with boundary values belonging to the boundary set.

To put the Strong Trace Property into perspective, we recall that it is highly nontrivial and known only in limited settings.
Strong trace properties were first shown for multi-dimensional scalar conservation laws in \cite{trace1}, 
and the technique was later used to extract additional structural information on solutions (see \cite{tracestruct1, tracestruct2}). 
However, whether general bounded entropy solutions satisfy the Strong Trace Property used here remains open. Weaker trace properties have been established for the isentropic Euler system with $\gamma=3$ in \cite{golding} and for genuinely nonlinear $2\times2$ systems under suitable structural assumptions in \cite{stpT,stpT2}. Related results on the structure of entropy dissipation and regularity for the isentropic Euler system with $\gamma=3$ were obtained in \cite{stpAMT}. The latter two works require entropy assumptions beyond the single physical entropy inequality imposed here.\\

\noindent {\bf $\bullet$ $L^2$ stability theory.} 
For a long time, $L^1$ was regarded as the natural space for studying the stability of shocks; see, for example, \cite{l1stability,l1stability2}. To treat a broader class of weak solutions beyond the small BV setting, an $L^2$-based framework known as $a$-contraction with shifts has been developed. In this framework, the perturbed solution is not required to have bounded variation. A fundamental obstacle to establishing $L^2$-stability is illustrated by the example in \cite[Section~1.1]{CKF25}: even a small perturbation of a shock may produce a large $L^2$-error due to the resulting displacement of the shock.

However, if one allows the shock to move by an appropriate shift, $L^{2}$ stability can be recovered. The idea of introducing a shift was first proposed by Vasseur \cite{Vshift}, and Leger later proved $L^{2}$ contraction for a single shock up to a shift \cite{Lshift} in the scalar case. For systems, the contraction property generally fails (see \cite{weighta2}), but it can be restored by introducing a suitable 
weight, as done in \cite{weighta}.

With the help of the shift and the weight, Chen, Krupa, and Vasseur \cite{qualitative} established the existence and uniqueness result of small BV solutions for general $2\times 2$ hyperbolic systems under the assumption of the existence of a strictly convex entropy. More recently, Chen, Faile, and Krupa \cite{CKF25} strengthened this by deriving a quantitative estimate. Our paper can be viewed as an extension of Chen-Faile-Krupa \cite{CKF25} to the initial-boundary setting. The remaining impermeable case is treated in \cite{Impermeable}. In parallel with the weak-BV stability theory, the stability of small BV solutions among inviscid limits has also been investigated. In 2024, Chen, Kang, and Vasseur \cite{NStoE} proved such a stability result in the setting of the isentropic Euler system and the barotropic Navier-Stokes equations. This result was subsequently extended to certain subsonic boundary configurations, including the impermeable wall problem in \cite{Impermeable} and the outflow problem in \cite{prepinviscidbv}.

Both \cite{qualitative} and \cite{CKF25} use the front-tracking method to construct piecewise constant approximate solutions $V_{\nu}$ that converge to the small BV solution $V$.
However, there is in general no $L^{2}$ contraction between $U \in \mathcal{S}_{\mathrm{weak}}$ and $V_{\nu}$. 
To obtain an $L^{2}$ contraction, the speeds of the shocks in the approximate solution must be modified artificially rather than following their Rankine-Hugoniot speeds, and one instead constructs shifted approximate solutions $\psi_{\nu}$. The dissipative estimate generated by these shifts along the shocks then enables the $L^{2}$ contraction between $U$ and $\psi_{\nu}$.

Since the shock speeds in the approximate solution $V_{\nu}$ and in the shifted approximate solution $\psi_{\nu}$ are different, the graphs of $V_{\nu}$ and $\psi_{\nu}$ generally do not coincide, even when they start from the same initial data.
Chen, Krupa, and Vasseur \cite{qualitative} first proved only a qualitative result, because the front-tracking method they used did not allow them to compare $\psi_{\nu}$ and $V_{\nu}$ directly. Instead, they relied on an additional uniqueness theory under the Bounded Variation Condition to show that the limits of $\psi_{\nu}$ is the small BV solution.

Chen, Faile, and Krupa \cite{CKF25} instead used the front-tracking semigroup framework of Bressan and Colombo \cite{BC95}, which relies on a carefully designed approximate Riemann solver that incorporates interpolated rarefaction and shock curves with artificially assigned wave speeds. These ingredients yield a Lipschitz continuous semigroup in $L^{1}$, which in turn allows one to compare the shifted approximate solution $\psi_{\nu}$ with the approximate solution $V_{\nu}$, even when their graphs do not align. Combined with a sharper quantitative dissipation estimate for shocks, Chen, Faile, and Krupa \cite{CKF25} established a quantitative contraction estimate, thereby recovering and strengthening the qualitative result of Chen, Krupa, and Vasseur \cite{qualitative}.

\subsection{Structure of the paper} 
In Section \ref{sec dissipation}, we recall the quantitative dissipation estimates for shocks and rarefactions, given in Proposition~\ref{Prop:S} and Proposition~\ref{Prop:R} (corresponding to Propositions~3.1 and~3.3 in~\cite{CKF25}), and establish boundary dissipation estimates in Lemma~\ref{lem: inflow boundary} and Lemma~\ref{lem: ouflow boundary}.
These estimates allow us to control the boundary contribution associated with the boundary set for $\mathcal{S}_{\mathrm{weak}}$.

The modified front-tracking method is introduced in Section~\ref{sec:MFTA} to construct piecewise-constant approximate solutions that converge to the small~BV solution. 
We consider two types of approximate solutions as in \cite{qualitative,CKF25}: the unshifted ones, as in \cite{AC97}, which generate a Lipschitz semigroup that is contractive in a metric $d_\nu$ equivalent to $L^1$; and the shifted ones, in which shocks move with the speeds prescribed in Proposition~\ref{Prop:S}, enabling a contraction estimate between $U \in \mathcal{S}_{\mathrm{weak}}$ and shifted approximate solutions later on.

In Section~\ref{sec combination and separation}, we show that the limit of the unshifted approximate solutions satisfying the inflow/outflow boundary condition \eqref{small BV BC}, up to a subsequence, also satisfies (\ref{small BV BC}).
The key observation underlying this conclusion is that the strength of each wave remains bounded by its initial strength at the time of generation (see Lemma~\ref{lem:ithbd}).
Since shocks may merge and rarefactions may be partitioned over time, we introduce a labeling system and a function associated with each label to track the evolution of wave strength.
Finally, we construct the weight function $a$ and prove the main Theorem, Theorem \ref{thm}, in Section~\ref{sec: proof of them}. 

In summary, the goal of this paper is to extend the argument of \cite{CKF25} to the initial-boundary value problem. 
The main analytical challenges are 
to establish the corresponding boundary dissipation estimate (see Section \ref{subsec boundary dissipation estimate}); and, in addition, to show that the limit of the unshifted approximate solutions satisfies the inflow/outflow boundary condition (\ref{small BV BC}) assigned for the small BV solutions, which is achieved in Section \ref{sec combination and separation}.

\subsection{Idea of proof}
\subsubsection{Stability of a reflected shock}
First, let us illustrate the idea of the proof with a single reflected shock as a simple example. We only illustrate the outflow problem, that is a more complicated case.

Let $U_b \in \Omega_{\mathrm{sub}}^{\mathrm{out}}$, and let the initial datum $U_0$ be a small $1$-shock located at $x_0 \in \mathbb{R}_+$, connecting $U_1$ and $U_R$, with Rankine-Hugoniot speed $\sigma_1<0$. Assume that $U_1$ satisfies the outflow boundary condition \eqref{small BV BC}, and that $U_1$ and $U_R$ are close to $U_b$.

Since we consider small BV solutions in the subsonic region, the $1$-shock reaches the boundary at some time $t_0>0$ and is reflected as a $2$-shock, with Rankine-Hugoniot speed $\sigma_2>0$, connecting $U_2$ and $U_R$, where $U_2$ satisfies the outflow boundary condition \eqref{small BV BC}.

Our small BV solution $V$ is 
\begin{equation*}
    V(x,t) = 
    \begin{cases}
        U_1 \, , \; \quad \textup{ if } x<x_0+\sigma_1 t \, , \; t<t_0 \, , \\
        U_2 \, , \; \quad \textup{ if } x<\sigma_2 t \, , \; t>t_0 \, , \\
        U_R \, , \; \quad \textup{ otherwise.}
    \end{cases}
\end{equation*}
We fix $T,R>0$ and focus on the cone of information 
\[
\{(x,t)\, : 0<t<T, \,\, 0<x<R-s(t-T)\},
\]
where $s$, defined in \eqref{def of s}, is a constant chosen strictly larger than the speed of any wave in the shifted and unshifted approximate solutions, and large enough so that the relative entropy flux satisfies \eqref{q controlled by eta}. For simplicity, we assume that the shock $V$ remains inside the cone of information.

\textbf{Step 1.} The dissipation calculations:
Let $U \in \mathcal{S}_{\mathrm{weak}}$. We first estimate the difference between the weak solution $U$ and the shifted approximate solution $\psi$ using the $a$-contraction property.

Thanks to the Strong Trace Property, for any constants $a_1^i, a_2>0$ and any Lipschitz function $h_i(t)$, we have
\begin{align*}
    &\frac{d}{dt} \left(\int_0^{h_i(t)} a_1^i \, \eta(U|U_i)\,dx + \int_{h_i(t)}^{R-s(t- T)} a_2 \, \eta(U|U_R)\,dx \right)\\
    &\leq - a_1^i \Big[ q(U(h_i(t)-,t);U_i) - \dot{h}_i(t)\eta(U(h_i(t)-,t)|U_i)\Big] + a_1^i q(U(0+,t);U_i) \\
    &\quad -a_2 \Big[ q(U((R-s(t-T)-,t);U_R) + s\eta(U(R-s(t-T)-,t)|U_R)\Big] \\
    &\quad +a_2 \Big[ q(U(h_i(t)+,t);U_R) - \dot{h}_i(t)\eta(U(h(t)+,t)|U_R)\Big]\\
    &=: A_1 + A_2 + A_3 + A_4.
\end{align*}
By Lemma~\ref{lem: ouflow boundary}, which provides the boundary dissipation estimates for the outflow case based on the boundary set in Definition~\ref{def:BVP}, we obtain $A_2 \le 0$. The geometry of the set $\mathcal{B}_{out}(U_b)$ as in \eqref{def:Bout} is quite natural and crucial for the proof of Lemma~\ref{lem: ouflow boundary}. However, proving Lemma~\ref{lem: ouflow boundary} is highly non-trivial. Indeed, to show 
$q(U;\Util) \leq 0$ for any $\Util \in \mathcal{B}_{out}(U_b) \cap \Omega_{sub}^{out}$ and $U \in \mathcal{B}_{out}(U_b)$, we separate the boundary set $\mathcal{B}_{out}(U_b)$ into four parts $\mathcal{B}_1,..., \mathcal{B}_4$, and use different strategy for each case $U\in \mathcal{B}_i$. 
The key idea is to use the relative positions of $U$ and $\Util$ for the first case $U\in \mathcal{B}_1 := \mathcal{B}_{out}(U_b) \cap \{u \geq u_b\}$. On the other hand, the other cases $U\in \mathcal{B}_i$, $i=2,3,4$, are handled by using the entropy condition of Hugoniot curve or rarefaction curve.

We have $A_3 \le 0$, by choosing $s$ to satisfy \eqref{q controlled by eta}.

Proposition~\ref{Prop:S} provides the dissipation estimates for shocks. It shows that if we fix $a_2=1$ and choose $a_1^i$ sufficiently close to $a_2$, as in \eqref{prop: s weight}, and if the $i$-shock is propagated with a shift $h_i$, then there exists a constant $K>0$ such that
\[
A_1 + A_4 \le - K \, |U_i - U_R| \, \bigl|\dot h_i(t) - \sigma_i\bigr|^2 .
\]
Hence, we define the weight define the shifted approximate $\psi$ as 
\begin{equation*}
    a(x,t) = 
    \begin{cases}
        a_1^1 \, , \; \quad \textup{ if } x<h_1(t) \, , \; t<t_1 \, , \\
        a_1^2 \, , \; \quad \textup{ if } x<h_2(t) \, , \; t>t_1 \, , \\
        a_2 \, , \; \quad \textup{ otherwise,}
    \end{cases}
    \psi(x,t) = 
    \begin{cases}
        U_1 \, , \; \quad \textup{ if } x<h_1(t) \, , \; t<t_1 \, , \\
        U_2 \, , \; \quad \textup{ if } x<h_2(t) \, , \; t>t_1 \, , \\
        U_R \, , \; \quad \textup{ otherwise,}
    \end{cases}
\end{equation*}
where $t_1$ is the time at which $h_1$ becomes $0$. Note that by \eqref{bound of h 1} and \eqref{bound of h 2}, such a $0 < t_1 < \infty$ exists.

We get
\begin{equation}\label{est:L2dis}
    \begin{aligned}
    &\int_0^R a(\cdot,T)\eta(U|\psi)(\cdot,T)\,dx - \int_0^{R+sT} a(\cdot,0)\eta(U|\psi)(\cdot,0)\,dx\\
    &\leq - K \left(\int_0^{t_1} |U_1 - U_R||\dot{h}_1(t) - \sigma_{1}|^2\,dt + \int_{t_1}^{T} |U_2 - U_R||\dot{h}_2(t) - \sigma_{2}|^2\,dt \right)=:- K \mathcal{D}.
\end{aligned}
\end{equation}

\textbf{Step 2.}  Calculation of the amount of shifting: The parameter $\nu>0$ is associated with the front-tracking construction and tends to zero in the limit, with smaller values of $\nu$ yielding more accurate approximations. By \eqref{def of phi ft}, \eqref{interpolated curve}, and \eqref{speed of shock unshifted}, the unshifted approximate solution coincides exactly with the shock $V$ when $\nu \ll 1$.
Note, however, that in general the unshifted approximate solutions depend on $\nu$.

Lemma~\ref{lem: L1 difference} estimates the difference between the unshifted approximate solution $V$ and the shifted approximate solution $\psi$ through a distance function $d_\nu$ equivalent to $L^1$ norm, showing that for some constant $C>0$,
\begin{align*}
    d_\nu(\psi,V)(T) 
    \leq C\left(\int_{0}^{t_1} |U_1 - U_R||\dot{h}_1(t) - \sigma_1|\,dt + \int_{t_1}^T |U_2 - U_R||\dot{h}_2(t) - \sigma_2|\,dt \right),
\end{align*}
Together with H\"older's inequality, we get
\begin{equation}\label{est:L1s}
    \begin{aligned}
    \|\psi(\cdot,T)-V(\cdot,T)\|_{L^1(0,R)} \leq C d_\nu(V,\psi)(T) \leq CT^{1/2}(|U_1-U_R| + |U_2 - U_R|)^{1/2}\mathcal{D}^{1/2}.
\end{aligned}
\end{equation}

\textbf{Step 3.} Putting everything together:
Recall that the state space $\mathcal{V}_0$ is bounded. By \eqref{est:L2dis}, \eqref{est:L1s} and H\"older's inequality, we obtain
\begin{align*}
    \|U(\cdot,T) - V(\cdot,T)\|_{L^2(0,R)}^2
    &\leq C\|U(\cdot,T) - V(\cdot,T)\|_{L^1(0,R)} \\
    &\leq \sqrt{R}\|U(\cdot,T) - \psi(\cdot,T)\|_{L^2(0,R)} + \|\psi(\cdot,T) - V(\cdot,T)\|_{L^1(0,R)}\\
    &\leq C_{R,T}\sqrt{\int_0^{R+sT} a(\cdot,0)\eta(U|\psi)(\cdot,0)\,dx} \\
    &= C_{R,T}\sqrt{\int_0^{R+sT} a(\cdot,0)\eta(U|V)(\cdot,0)\,dx} \\
    &\leq C_{R,T} \|U(\cdot,0) - V(\cdot,0)\|_{L^2(R+sT)}.
\end{align*}
Hence, we obtain the inequality \eqref{thm eq}.
\begin{figure}[h!]
    \centering
    \begin{subfigure}[b]{0.45\textwidth}
        \centering
        \includegraphics[width=\textwidth]{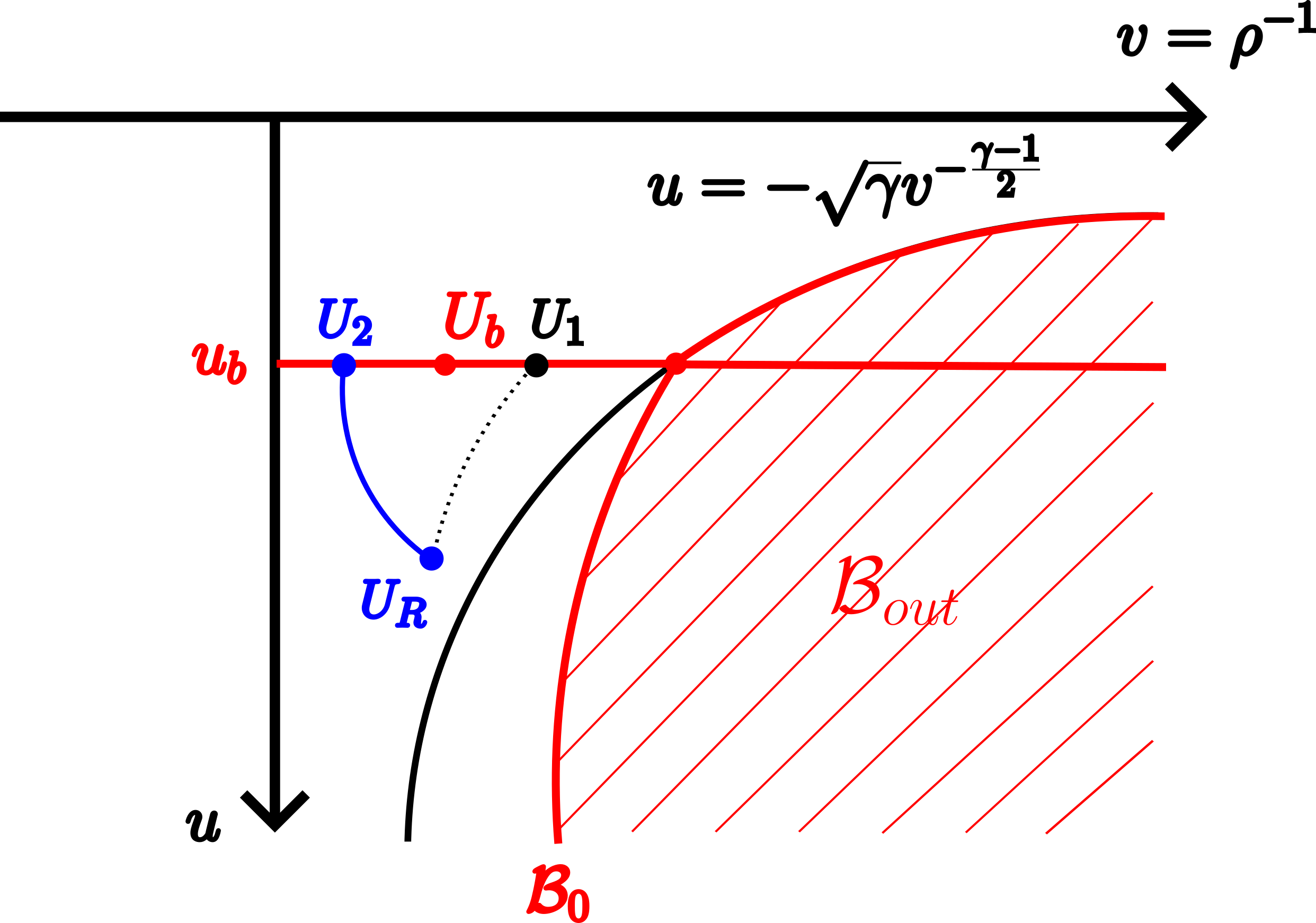}
        \caption{Reflected shock waves}
    \end{subfigure}
    \quad
    \begin{subfigure}[b]{0.45\textwidth}
        \centering
        \includegraphics[width=\textwidth]{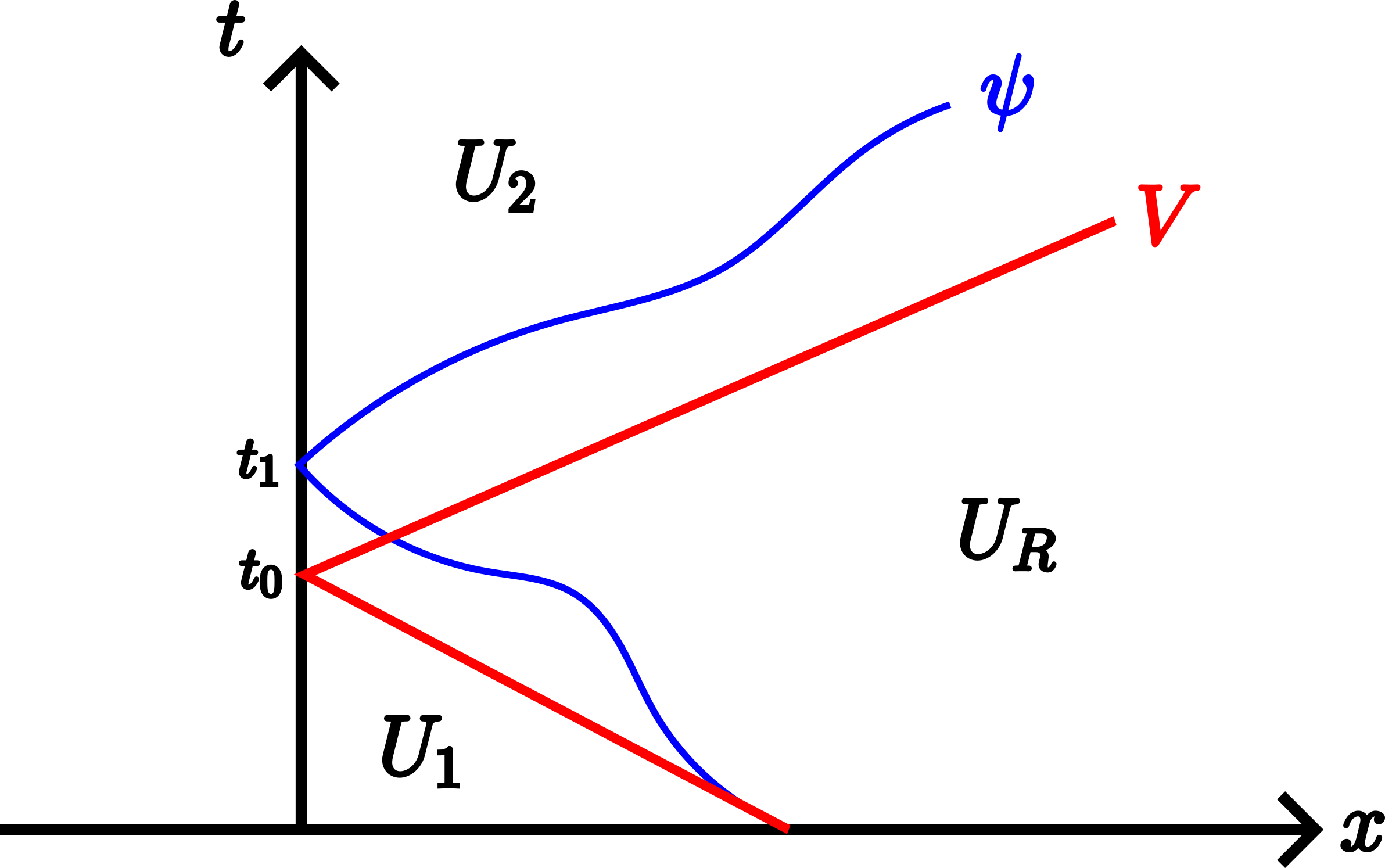}
        \caption{Unshifted and shifted shock waves}
    \end{subfigure}
    \caption{Reflected shock wave}
    \label{fig:RRS}
\end{figure}

\subsubsection{Boundary condition satisfied by limits of approximate solutions}
A major new difficulty in the boundary value problem, compared to the Cauchy problem, is that although every unshifted approximate solution satisfies the inflow/outflow boundary condition \eqref{small BV BC}, it is highly nontrivial to show that this property is preserved in the limit.
The difficulty is that wave strengths are not preserved by interaction or reflection at the boundary: shocks may merge and rarefactions may be partitioned. So, the total variation crossing a fixed vertical line cannot be uniformly estimated by simply tracking strengths of fronts.
To overcome this difficulty, we introduce a labeling system that associates each wave with a collection of elements carrying tracked strengths in Section \ref{subsec label}. 
We keep track of the strength of each element instead of each front, according to simple updating rules at every event (interaction or boundary reflection), so that each wave strength of the approximate solution is always recovered as the sum of the tracked strengths of its associated elements. The crucial estimate is that each tracked strength never exceeds, up to a uniform multiplicative constant, its previous strength at every event, as in Lemma \ref{lem:iint} and \eqref{est:stpm}.
Since each element can cross any fixed vertical line at most twice during its lifetime, summing over all elements yields a uniform bound on the total variation of the approximate solutions $V_\nu(x_*,\cdot)$ along every vertical line $x=x_*$. 
On top of that, since the approximate solutions remain in a sufficiently small neighborhood of the subsonic state $U_b$, every $1$-wave propagates with an $O(1)$-negative speed, while every $2$-wave propagates with an $O(1)$-positive speed, uniformly in $\nu$. 
Thus, the spatial derivatives of approximate solutions admit a uniform bound in the space of finite Radon measures on time.
Therefore, the above uniform boundedness provides strong convergence in $C_x([0,R];L_t^1([0,T]))$.  
As a consequence, the inflow/outflow boundary condition is preserved in the limit as in Proposition \ref{lem:BC}.

\section{Boundary sets}\label{sec:bdset}
In this section, we provide the necessary and sufficient conditions for the existence of boundary layer solutions. We then characterize the boundary set $\mathcal{B}_{in/out}$ and derive expressions for $\mathcal{B}_{in/out}$.
\subsection{Boundary layer solutions}
For given two state $U_- = (\rho_-,\rho_- u_-)$ and $U_+ = (\rho_+, \rho_+ u_+)$, the boundary layer (BL) solution $U^{BL}(x;U_-,U_+):=(\rho^{BL}(x;U_-,U_+),\rho^{BL} u^{BL}(x;U_-,U_+))$ is defined as a solution to
	\begin{equation}\label{eq:BL}
	\begin{aligned}
		&\left(\rho u \right)' = 0, \quad x\geq 0, \quad '=\frac{d}{dx},\\
		&\left(\rho u^2 + p(\rho)\right)' = u'',\\
		&U(0)=U_-, \quad U(+\infty) = U_+. 
	\end{aligned}
\end{equation}
From \eqref{eq:BL}$_1$, it follows immediately that the momentum of the BL solution is conserved: 
\[
\rho(x) u(x)  = \rho_+ u_+ = \rho_- u_-.
\]
In other words, for given boundary state $U_-$, the state $U_+$ should lie on the curve $BL(U_-)$ defined by
\begin{equation*}
	BL(U_-):=\left\{U=(\rho,\r u) \in \Omega : u = \frac{\rho_- u_- }{\rho} \right\}.
\end{equation*}
In what follows, we provide the necessary and sufficient conditions for the existence of the BL solution in both the outflow and inflow settings. To this end, we first define the Hugoniot locus associated with the $i$-th characteristic field centered at $U_-$:
\begin{equation}\label{def:Hug}
S_i(U_-) := \left\{U\,:\, u = u_- +(-1)^{i} (\rho-\rho_-)\left|\frac{p(\rho) - p(\rho_-)}{\rho \rho_- (\rho - \rho_-)}\right|^{1/2}, \, \, u \leq u_- \right\}.
\end{equation}
\begin{rem}
    Although the right end state is often used as the reference state for wave curves in problems on the half-line, we consistently use the left end state as the reference state for both the Hugoniot loci and the boundary layer curves. This convention is particularly convenient for the boundary dissipation estimates established later in Section~\ref{subsec boundary dissipation estimate}.
\end{rem}
The following lemma characterizes the existence of BL solutions for the outflow problem. It is a reformulation of \cite[Proposition~4.1]{M01}, with the left end state taken as the reference state, whereas \cite[Proposition~4.1]{M01} uses the right end state.
\begin{lem}\label{lem:BLout}
(Existence of BL solution in the outflow setting) Consider the case $u_-<0$ with $U_+ \in BL(U_-)$. Then, the BL solution connecting $U_-$ to $U_+$ exists uniquely if and only if one of the following holds:
\begin{enumerate}
	\item[(i)] $U_- \in \Omega_{super}^-$ and 
	\[
		u_* \geq u_+ > u_- \,\, \text{or } u_+ < u_-,
	\]
	where $U_* := (\rho_*,
    \r_* u_*)$ denotes the unique intersection point of the curve $BL(U_-)$ with the transonic curve $\Gamma_{trans}^-$.
	\item[(ii)] $U_- \in \Gamma_{trans}^{-}$ and $u_+ < u_- < 0$.
	\item[(iii)] $U_- \in \Omega_{sub}^-$ and $u_+ < \check{u}<0$, where $\check{U}=(\check{\rho},\check{\rho} \check{u})\neq U_-$ denotes the (unique) intersection point of 
    $S_2(U_-)$ and the curve $BL(U_-)$. 
\end{enumerate}
\end{lem}
In the case of the inflow problem, the existence of boundary layer solutions is characterized by the following lemma.
\begin{lem}\label{lem:BLin}\cite[Proposition 4.2]{M01}
(Existence of BL solution in the inflow setting)
Suppose that $u_->0$ and $U_+ \in BL(U_-)$. Then, the BL solution connecting $U_-$ to $U_+$ exists uniquely if and only if one of the following holds:
\begin{enumerate}
	\item $U_- \in \Omega_{super}^+$ and $0<u_+ < \check{u}$, where $\check{U}=(\check{\rho},\check{\r}\check{u})\neq U_-$ denotes the (unique) intersection point of 
    $S_1(U_-)$ and the curve $BL(U_-)$.
	\item $U_- \in \Gamma_{trans}^+$ and $0<u_+<u_-$.
	\item $U_- \in \Omega_{sub}^+$ and $0<u_+ < u_*$, where $U_*:=(\rho_*,\r_* u_*)$ denotes the unique intersection point of the line $BL(U_-)$ with the transonic curve $\Gamma_{trans}^+$.
\end{enumerate}
\end{lem}

\subsection{Euler boundary sets $\mathcal{B}_{in/out}(U_b)$}
Based on Lemma~\ref{lem:BLin} and Lemma~\ref{lem:BLout}, we can explicitly describe the boundary sets $\mathcal{B}_{in/out}(U_b)$ (see figure~\ref{fig:BC}).

Let $U_b \in \Omega_{\mathrm{sub}}^{\mathrm{in/out}}$.  For the inflow case, the 
boundary set $\mathcal{B}_{\mathrm{in}}(U_b)$ can be described as
\[
\mathcal{B}_{\mathrm{in}}(U_b)
   = BL(U_b) \cap (\Omega_{\mathrm{sub}}^{+} \cup \Gamma_{\mathrm{trans}}^{+}).
\]
For the outflow case, Definition~\ref{def:BVP} provides the formulation
\begin{equation}\label{def:Bout}
    \mathcal{B}_{out}(U_b) = \{U:u=u_b\} \cup \{U: u \geq u_b, \, \, U \in \Gamma_{trans}^- \cup \Omega_{super}^-\}  \cup \mathcal{B}_0(U_b) \cup \mathcal{B}_-(U_b),
\end{equation}
where $\mathcal{B}_0(U_b)$ and $\mathcal{B}_-(U_b)$ are defined by 
\[
\mathcal{B}_0(U_b) := \{U: \exists \, U_* \, \, \text{such that } U \in S_2(U_*), \, \, s_2(U_*,U) = 0,\,\, u_* = u_b \},
\]
and
\[
\mathcal{B}_-(U_b) := \{U: \exists \, U_* \, \, \text{such that } U \in S_2(U_*), \, \, s_2(U_*,U) < 0,\,\, u_* = u_b \}.
\]
Here, $s_2(U_*,U)$ denotes the speed of a 2-shock wave connecting $U_*$ (left) and $U$ (right), and $S_2(U_*)$ is the Hugoniot locus associated with the second characteristic field centered at $U_*$, defined in \eqref{def:Hug}.

\begin{figure}[h!]
    \centering
    \begin{subfigure}[b]{0.45\textwidth}
        \centering
        \includegraphics[width=\textwidth]{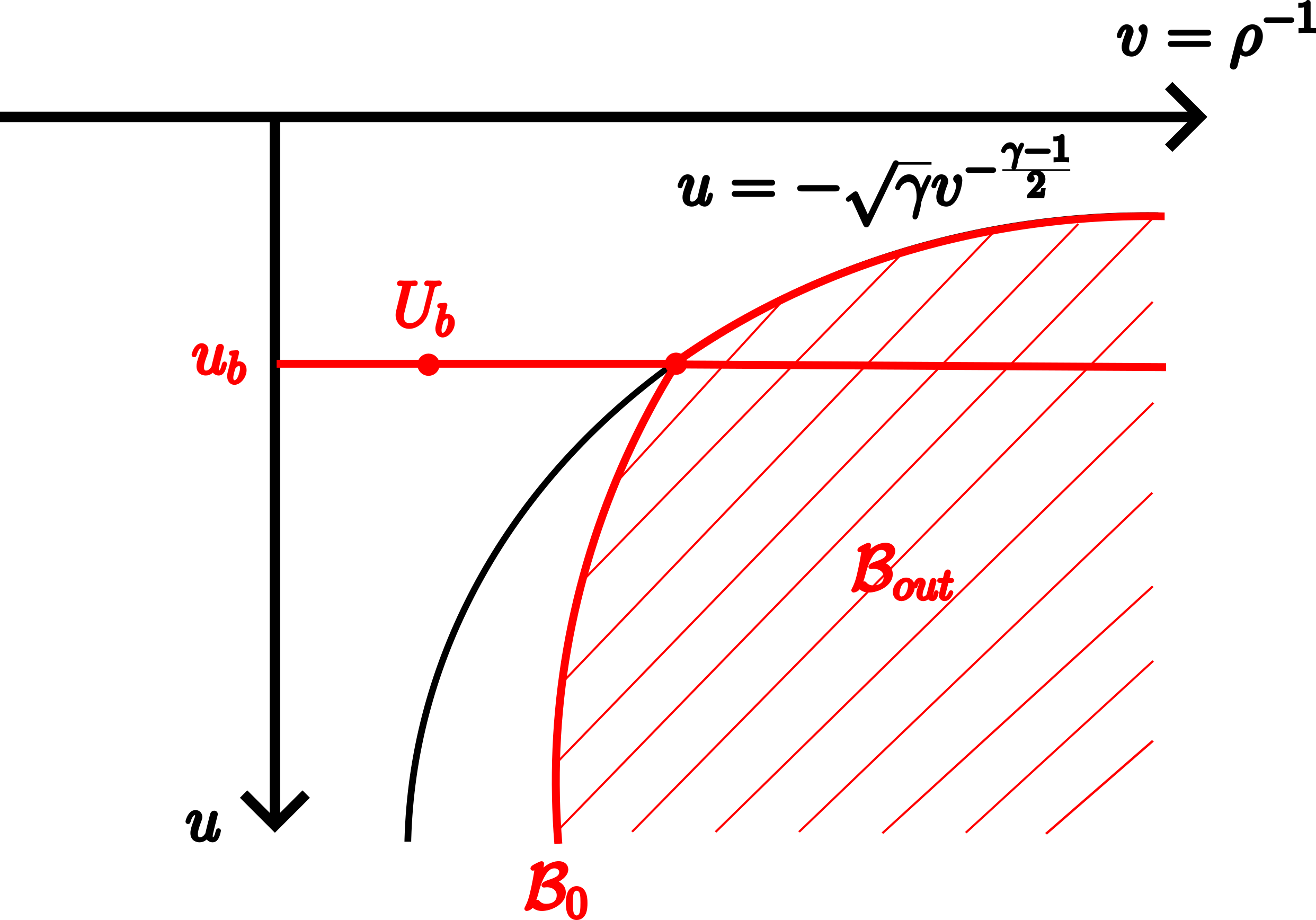}
        \caption{Outflow boundary set \( \mathcal{B}_{\mathrm{out}} \)}
    \end{subfigure}
    \quad
    \begin{subfigure}[b]{0.45\textwidth}
        \centering
        \includegraphics[width=\textwidth]{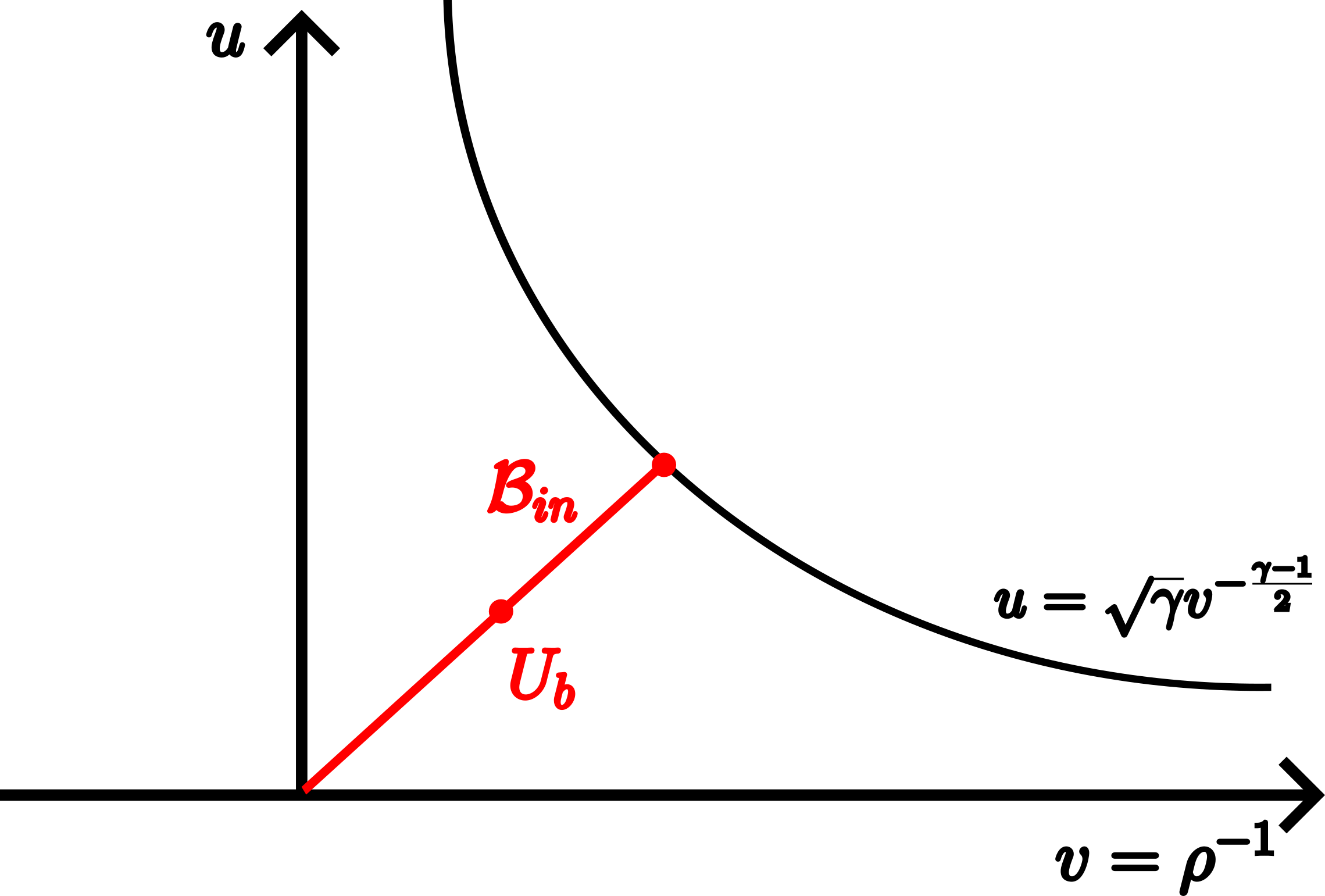}
        \caption{Inflow boundary set \( \mathcal{B}_{\mathrm{in}} \)}
    \end{subfigure}
    \caption{The figures above illustrate the boundary sets for the outflow and inflow cases in $(v=\rho^{-1},u)$ coordinates}
    \label{fig:BC}
\end{figure}
\vspace{0.2cm}
Suppose that the small BV solution, say $V$, remains in a small neighborhood of $U_b$, which is contained in the subsonic region $\Omega_{sub}$. Then, there exists a positive constant $c$ (depending on $U_b$) such that
\[
\lambda_1(V(x,t)) \leq -c \leq 0 \leq c \leq \lambda_2(V(x,t)), \,\, \text{for a.e. } (x,t) \in \mathbb{R}_+ \times \mathbb{R}_+.
\] 
In this regime, the boundary sets for the isentropic Euler system form one-dimensional manifolds as follows: for a.e. $t>0$,
\begin{align}\label{small BV BC}
    u(0,t) = u_b < 0 \, \textup{ for the outflow and } \r(0,t)u(0,t) = \r_b u_b > 0 \, \textup{ for the inflow.}
\end{align}
These cases correspond to particular boundary conditions within the (NC) in Amadori \cite{A97}. Therefore, to construct the small BV solutions, one can use the front tracking algorithm for the (NC) case in Amadori \cite{A97} and Amadori and Colombo \cite{AC97}.

However, note that large and wild solutions $U \in \mathcal{S}_{weak}$ may lie in transonic or supersonic regions. Consequently, for such solutions, one must consider the full boundary sets $\mathcal{B}_{out}(U_b)$ and $\mathcal{B}_{in}(U_b)$.

\section{Dissipation estimates}\label{sec dissipation}
\setcounter{equation}{0}
We apply the relative entropy method introduced by Dafermos \cite{remDaf} and DiPerna \cite{remDi}. Given a convex entropy function $\eta$, we define the associated pseudo-distance between any $a,b \in \mathcal{V}_0 \times \mathcal{V}$ by
\begin{align*}
    \eta(a|b)=\eta(a)-\eta(b)-\nabla\eta(b)\cdot(b-a) \, . 
\end{align*}
We call this quantity the relative entropy of $a$ with respect to $b$. Since $\mathcal{V}_0$ is bounded, the relative entropy $\eta(a \mid b)$ is equivalent to $|a-b|^2$.
Correspondingly, we define the relative entropy flux by, for any $a,b \in \mathbb{R}^2$,
\begin{align*}
    q(a;b)=q(a)-q(b)-\nabla \eta(b)\cdot(F(a)-F(b)) \,,
\end{align*}
where $F(U)$ denotes the flux of the system \eqref{eq:E}.

If $U$ is a weak solution of \eqref{eq:E} satisfying the entropy inequality \eqref{entropy inequality}, then for any constant state $b \in \mathcal{V}$, we have
\begin{align*}
    (\eta(U \mid b))_t + (q(U ; b))_x \leq 0 \, .
\end{align*}
The system \eqref{eq:E} admits an entropy 
\[
\eta(U) = \frac{1}{2}\rho u^2 + \frac{\rho^{\gamma}}{\gamma - 1}.
\]
The corresponding relative entropy and relative entropy flux are given by 
\begin{align*}
    &\eta(U_1|U_2) = \frac{1}{2}\rho_1 |u_1  - u_2|^2 + \frac{p(\rho_1|\rho_2)}{\gamma - 1},\\
    &q(U_1;U_2) = u_1 \eta(U_1|U_2) + (p(\rho_1) - p(\rho_2))(u_1 - u_2),
\end{align*} 
where $U_1 = (\rho_1,\rho_1 u_1), U_2 = (\rho_2,\rho_2 u_2)$ and the relative pressure $p(\rho_1|\rho_2)$ is defined by 
\[
p(\rho_1|\rho_2) = p(\rho_1) - p(\rho_2) - p'(\rho_2)(\rho_1 - \rho_2).
\]
In Section~\ref{subsec: dissipation for S R}, we recall the contraction property of the relative entropy for both shock and rarefaction fans established in \cite{CKF25}. The contraction property for rarefaction fans, even in multi-dimensional
settings, can also be found in \cite{qualitative,disR1,disR2}.
In Section~\ref{subsec boundary dissipation estimate}, we prove that the relative entropy flux between $U \in \mathcal{S}_{weak}$ and a small BV solution $V$ is non-positive at the boundary. All these properties will be used to carry out the contraction argument in Section \ref{sec: proof of them}.

\subsection{Quantitative dissipation estimates for shocks and rarefactions}\label{subsec: dissipation for S R}
\begin{prop}\cite[Proposition 3.1]{CKF25}\label{Prop:S}
	Let $d \in \mathcal{V}$. There exist constants $\hat{\lambda}, C_1, K, \varepsilon_0 >0$ such that the following holds. Consider any $1$-shock or $2$-shock $(U_L, U_R, \sigma_{LR})$ with shock strength $s_0 := |U_L - U_R|$. Moreover, assume 
    \[
     U_L, U_R \in B_{\e_0}(d) \subset \Omega_{sub} \quad \text{and} \quad  \sup_{U \in B_{{2\e_0}}(d)}\lambda_1(U) <0< \inf_{U \in B_{{2\e_0}(d)}}\lambda_2(U). 
    \]
    Then, for any $a_1, a_2 > 0$ satisfying 
    \begin{equation}\label{prop: s weight}
        \begin{aligned}
            &1- 2C_1 s_0 \leq \frac{a_2}{a_1} \leq 1 - \frac{C_1 s_0}{2} \quad \text{if } (U_L, U_R) \, \, \text{is a } 1\text{-shock},\\
		    &1 +\frac{C_1 s_0}{2} \leq \frac{a_2}{a_1} \leq 1+ 2C_1 s_0 \quad \text{if } (U_L, U_R) \, \, \text{is a } 2\text{-shock},\\
        \end{aligned}
    \end{equation}
	we obtain the following: for any $U \in \mathcal{S}_{weak}, T>0, T_{start} \in [0,T]$, and $x_0 \in \mathbb{R}_+$, there exists a Lipschitz shift function $h:[T_{start}, T] \to \mathbb{R}$ with $h(T_{start}) = x_0$ such that the dissapation functional satisfies the bound: 
	\begin{align*}
		&a_2 \Big[ q(U(h(t)+,t);U_R) - \dot{h}(t)\eta(U(h(t)+,t)|U_R)\Big]\\
		&\qquad  - a_1 \Big[ q(U(h(t)-,t);U_R) - \dot{h}(t)\eta(U(h(t)-,t)|U_R)\Big] \leq -a_2 K s_0 (\dot{h}(t) - \sigma_{LR})^2,
	\end{align*}
	for almost every $t \in [T_{start},T]$.

	In addition, if $(U_L, U_R)$ is a $1$-shock, then
    \begin{equation}\label{bound of h 1}
        -\frac{\hat{\lambda}}{2} \leq \dot{h}(t) \leq \sup_{U \in B_{2\varepsilon_0}(d)}\lambda_1(U) <0, \quad \text{for a.e. } t \in [T_{start},T]. 
    \end{equation}
	Similarly, if $(U_L,U_R)$ is a $2$-shock wave, then
    \begin{equation}\label{bound of h 2}
        0< \inf_{U \in B_{2\varepsilon_0}(d)}\lambda_2(U) \leq \dot{h}(t) \leq \frac{\hat{\lambda}}{2}, \quad \text{for a.e. } t \in [T_{start},T].
    \end{equation}
\end{prop}

\begin{rem}
The shift $h$ is introduced in order to obtain the contraction estimate later in Section \ref{sec: proof of them}. The construction of the corresponding shifted approximate solution is given in Section~\ref{sec:SFTA}.
\end{rem}

\begin{prop}\cite{qualitative}\label{Prop:R}
    There exists a constant $K>0$ such that the followings are true. For any rarefaction wave $U^R(y)$ with $v_L \leq y \leq v_R$, denote 
    \[
        \sigma^R := |v_L - v_R| + \sup_{v \in [v_L,v_R]}|U_L - U^R(y)|, \quad U^R(v_L) = U_L, \, \,  U^R(u_R) = U_R.
    \]
    Then, for any $U \in \mathcal{S}_{weak}$ and $v_L \leq v \leq v_R$, and any $t >0$, we have 
    \begin{equation}\label{prop R ineq}
    \begin{aligned}
          &\int_0^t \Big(q(U(\tau v+,\tau);U_R) - q(U(\tau v-,\tau);U_L) - v\big(\eta(U(\tau v+,\tau)|U_R) - \eta(U(\tau v-,\tau)|U_L) \big) \Big)d\tau\\
          &\qquad \leq K(\sigma^R)^2 t.
    \end{aligned}
    \end{equation}
\end{prop}
\subsection{Boundary dissipation estimates}\label{subsec boundary dissipation estimate}
We provide boundary dissipation estimates for the outflow and inflow cases, respectively, and ensure that, in the proof of the main theorem (Section \ref{subsec proof of them main theorem}), the boundary term $F_0^+$ vanishes.
\begin{lem}\label{lem: ouflow boundary}
(Boundary estimates for the outflow problem)

For any $U_b \in \Omega_{sub}^{out}$ with $u_b<0$, and for any constant states $\Util \in \mathcal{B}_{out}(U_b) \cap \Omega_{sub}^{out}$ and $U \in \mathcal{B}_{out}(U_b)$, we have 
\[
q(U;\Util) \leq 0.
\]
\end{lem}
\begin{proof}
	Recall from \eqref{def:Hug} that $S_2(\Util)$ denotes the Hugoniot locus associated with the second characteristic field centered at $\Util$, defined by
    \[
    S_2(\Util) := \left\{U\,:\, u = \util + (\rho-\rhotil)\left|\frac{p(\rho) - p(\rhotil)}{\rho \rhotil (\rho - \rhotil)}\right|^{1/2}, \, \, u \leq \util \right\}.
    \]
	We begin by splitting the set $\mathcal{B}_{out}(U_b)$ into four regions: $\mathcal{B}_{out}(U_b)= \mathcal{B}_1 \cup \mathcal{B}_2 \cup \mathcal{B}_3 \cup \mathcal{B}_4$, where
    \[
    \begin{aligned}
        &\mathcal{B}_1 := \mathcal{B}_{out}(U_b) \cap \{u \geq u_b\},\\
        &\mathcal{B}_2 := \mathcal{B}_{out}(U_b) \cap \left\{U\,:\, u_b > u > \util + (\rho-\rhotil)\left|\frac{p(\rho) - p(\rhotil)}{\rho \rhotil (\rho - \rhotil)}\right|^{1/2} \right\},\\
        &\mathcal{B}_3 := \mathcal{B}_{out}(U_b) \cap S_2(\Util),\\
        &\mathcal{B}_4 := \mathcal{B}_{out}(U_b) \cap \left\{U\,:\,u < \util + (\rho-\rhotil)\left|\frac{p(\rho) - p(\rhotil)}{\rho \rhotil (\rho - \rhotil)}\right|^{1/2} \right\}.
    \end{aligned}
    \]
    Roughly speaking, $\mathcal{B}_2$ and $\mathcal{B}_4$ correspond to the portions of $\mathcal{B}_{out}(U_b)$ lying above and below the 2-shock curve, respectively. In what follows, denote $U=(\r,\r u)$ and notice that $\widetilde{U} \in \mathcal{B}_{out}(U_b)\cap \Omega_{sub}^{out}$ implies $\tilde{u} = u_b$.
    
$\bullet$ \textit{Case I:} Suppose that $U \in \mathcal{B}_1$. First, using $u<0$, we have
\begin{align*}
	q(U;\Util) &= u\eta(U|\Util) + (p(\rho) - p(\rhotil))(u-\util)\\
	&\leq (p(\rho) - p(\rhotil))(u-\util).
\end{align*}
Since $U\in\mathcal{B}_1$ implies that either $u=u_b$, or $u>u_b$ and $\r < \r_b$, we have
\[
(p(\rho) - p(\rhotil))(u-\util) \leq 0 \, \,  \text{on } \mathcal{B}_1.
\] 
Thus, we obtain $q(U;\Util) \leq 0$.

$\bullet$ \textit{Case II:} Now, consider the case $U \in \mathcal{B}_2$. Then, there exists (unique) state $U_* = (\rho_*, \rho_* u_*) \in \{u=u_b\}$ such that $\rhotil > \rho_*$ and $U \in S_2(U_*)$. Indeed, $U_*$ is the unique intersection of the line $u=u_b$ with $S_2^r(U)$, where $S_2^r(U)$ denotes the 2-Hugoniot locus with $U$ taken as its right end state. From the entropy identity, we have 
\[
q(U;\Util) = q(U;U_*) + q(U_*;\Util) - (\nabla \eta(U_*) - \nabla \eta(\Util))\cdot (F(U_*) - F(U)) =: I_1 + I_2 + I_3.
\]
We now show that $I_i \leq 0$ for each $i=1,2,3$. 

Since $U_*$ (left) and $U$ (right) are connected by a 2-shock curve, the entropy inequality yields
\[
q(U;U_*) - q(U_*;U_*) \leq s (\eta(U|U_*) - \eta(U_*|U_*)),
\]
where $s$ denotes the shock speed. Since $s\leq0$ from $U \in \mathcal{B}_0(U_b) \cup \mathcal{B}_-(U_b)$ by \eqref{def:Bout}, it follows that
\[
I_1 = q(U;U_*) \leq s \eta(U|U_*) \leq 0.
\]
Next, since $u_* = \util = u_b$, we apply the Case I to obtain $I_2 = q(U_*;\Util) \leq 0$.

Thus, it remains to show $I_3 \leq 0$. Recall that
\begin{align*}
	\nabla \eta(U) = \begin{pmatrix}
		-\frac{u^2}{2} + \frac{\gamma}{\gamma - 1}\rho^{\gamma -1} \\ u
	\end{pmatrix}, \quad F(U) = \begin{pmatrix}
		\rho u \\ \rho u^2 + p(\rho)
	\end{pmatrix}.
\end{align*}
Using $u_* = \util = u_b$, we find that 
\begin{align*}
	I_3 &= - (\nabla \eta(U_*) - \nabla \eta(\Util))\cdot (F(U_*) - F(U))\\
	&=-\frac{\gamma}{\gamma -1}\big(\rho_*^{\gamma -1} - \rhotil^{\gamma -1}\big)(\rho_* u_* - \rho u).
\end{align*}
Since $U_*$ and $U$ are connected by a 2-shock curve with non-positive speed $s$, the Rankine-Hugoniot condition yields 
\[
\rho_* u_* - \rho u = s (\rho_* - \rho) \leq 0.
\]
Combining this with the fact that $\rho_* < \rhotil$, we deduce that $I_3 \leq 0$.

Therefore, we conclude that 
\[
q(U;\Util) = I_1 + I_2 + I_3 \leq 0.
\]

$\bullet$ \textit{Case III:} Suppose that $U \in \mathcal{B}_3$. Since $\Util$ (left) and $U$ (right) are connected by the 2-shock curve with non-positive speed $s$, the entropy inequality yields
\[
q(U;\Util) = q(U;\Util) - q(\Util;\Util) \leq s (\eta(U|\Util) - \eta(\Util|\Util)) \leq 0.
\]

$\bullet$ \textit{Case IV:} Finally, suppose that $U \in \mathcal{B}_4$. Then, consider the whole space Riemann problem connecting $\Util$ (left) to $U$ (right). From the assumption $U \in \mathcal{B}_4$, the corresponding Riemann solution is given by one of the following three types: a 1-shock; a composition of a 1-shock and a 2-shock; or a composition of a 1-shock and a 2-rarefaction.

(i) If $\Util$ and $U$ forms a 1-shock wave, then we have the entropy inequality:
\[
q(U;\Util) \leq s \eta(U|\Util) \leq 0,
\]
where $s<0$ is the speed of the 1-shock wave conntecting $\Util$ to $U$.

(ii) Suppose that $\Util$ and $U$ form a composite wave of a 1-shock and a 2-shock. Let $U_m$ denote the intermediate state of the corresponding Riemann solution. Then we have
\begin{align*}
	&q(U_m;\Util) - q(\Util;\Util) \leq s_1 \big(\eta(U_m|\Util) - \eta(\Util;\Util)\big),\\
	&q(U;\Util) - q(U_m;\Util) \leq s_2 \big(\eta(U|\Util) - \eta(U_m|\Util)\big),
\end{align*}
where $s_1$ and $s_2$ denote the speeds of the 1-shock and 2-shock waves, respectively. Adding the two inequalities, we find that
\[
q(U;\Util) \leq s_2 \eta(U|\Util) + (s_1 - s_2)\eta(U_m|\Util).
\]
Clearly, we have $s_1 <s_2$. We next show that $s_2 \leq 0$. To this end, consider the 2-Hugoniot locus $S_2^r(U)$ with $U$ taken as its right end state, and let $U_*$ denote of $S_2^r(U)$ with the line $u=u_b$. Since $U \in \mathcal{B}_4 \subset \mathcal{B}_0(U_b) \cup \mathcal{B}_-(U_b)$ by \eqref{def:Bout}, it follows that $s_2(U_*,U) \leq 0$, where $s_2(U_*,U)$ denotes the speed of the 2-shock connecting $U_*$ and $U$. Moreover, the explicit formula
\[
s_2(U_1,U) = u + \sqrt{\frac{\r_1(p(\r) - p(\r_1)}{\r(\r - \r_1)}},
\]
shows that $s_2(U_1,U)$ is strictly increasing with respect to $\rho_1$. Since $\rho_m < \rho_*$, we obtain 
\[
s_2 = s_2(U_m,U) < s_2(U_*,U) \leq 0.
\]
Therefore, we conclude that $q(U;\Util) \leq 0$.

(iii) Now, suppose that $\Util$ and $U$ form a composition of a 1-shock and a 2-rarefaction. Let $U_m$ denote the intermediate state. First, we have the entropy inequality:
\begin{equation}\label{est:1s}
	q(U_m;\Util) - q(\Util;\Util) \leq s_1 \big(\eta(U_m|\Util) - \eta(\Util;\Util)\big),
\end{equation}
where $s_1$ denote the 1-shock speed. Now, let $U^R\left(\frac{x}{t}\right)$ denote the 2-rarefaction wave connecting $U_m$ to $U$. Observe first that 
\begin{align*}
	\frac{d}{dt} \int_{\lambda_2(U_m)t}^{\lambda_2(U)t} \eta\left(U^R\left(\frac{x}{t}\right)|\Util \right)\,dx &= \frac{d}{dt} \int_{\lambda_2(U_m)}^{\lambda_2(U)} \eta\left(U^R(z)|\Util \right)t\,dz \\
	&= \int_{\lambda_2(U_m)}^{\lambda_2(U)} \eta\left(U^R(z)|\Util \right)\,dz >0.
\end{align*}
Applying the entropy inequality with respect to the entropy $\eta(\cdot|\Util)$ to the left-hand side of the above, we have
\begin{equation}\label{est:2r}
	q(U;\Util) - q(U_m;\Util) \leq \lambda_2(U)\eta(U|\Util) - \lambda_2(U_m)\eta(U_m|\Util).
\end{equation}
Summing the inequalities \eqref{est:1s} and \eqref{est:2r} yields
\[
q(U;\Util) \leq \lambda_2(U) \eta(U|\Util) + (s_1 - \lambda_2(U_m))\eta(U_m|\Util).
\]
Since $\lambda_2(U) \leq 0$ from $U \in \{u\leq-c\}$, and  $s_1 - \lambda_2(U_m) \leq 0$, we conclude that $q(U;\Util) \leq 0$.
\end{proof}

\begin{lem}\label{lem: inflow boundary}
(Boundary estimates for the inflow problem)

Fix $U_b \in \Omega_{sub}^{in}$ with $u_b>0$. For any constant states $\Util \in \mathcal{B}_{in}(U_b) \cap \Omega_{sub}^{in}$ and $U \in \mathcal{B}_{in}(U_b)$, we have 
\[
q(U;\Util) \leq 0.
\]
\end{lem}
\begin{proof}
Let $\Util = (\rhotil, \rhotil \util)$ and $U = (\rho, \rho u)$. Define $a>0$ by $\rho = a \rhotil$, and set $m := \rhotil \util = \rho u$. Then, we have 
\[
u = \frac{m}{a\rhotil}.
\] 
Moreover, observe that $U \in \{u=c\}$ if and only if $m^2 = \gamma a^{\gamma + 1}\rhotil^{\gamma + 1}$. Indeed,
\begin{equation}\label{est:atil}
	u = \sqrt{\gamma \rho^{\gamma - 1}} \iff \frac{m^2}{a^2 \rhotil^2} = \gamma a^{\gamma - 1}\rhotil^{\gamma -1} \iff m^2 = \gamma a^{\gamma + 1}\rhotil^{\gamma + 1}.
\end{equation}
Let $\tilde{a}$ denote the (unique) value sastisfying $m^2 = \gamma \tilde{a}^{\gamma + 1}\rhotil^{\gamma + 1}$. Then, we find that $0<\tilde{a}<1$, and $U \in \mathcal{B}_{in}(U_b)$ if and only if $a \in [\tilde{a},\infty)$.

Now, substituting the above notations into the relative entropy flux  $q(U;\Util)$, we obtain 
\begin{align*}
	&q(U;\Util) = \frac{\rho u }{2}|u-\util|^2 + \frac{u}{\gamma -1}p(\rho|\rhotil) + (p(\rho) - p(\rhotil))(u - \util)\\
	&\quad =\frac{m}{2}\left|\frac{m}{a \rhotil} - \frac{m}{\rhotil}\right|^2 + \frac{m}{(\gamma -1) a \rhotil}\big((a \rhotil)^\gamma - \rhotil^\gamma - \gamma \rhotil^{\gamma - 1}(a \rhotil - \rhotil)\big)
	+ \big((a \rhotil)^\gamma - \rhotil^\gamma)\left(\frac{m}{a \rhotil} - \frac{m}{\rhotil}\right)\\
	&\quad = \frac{m^3}{2\rhotil^2}\left(1-\frac{1}{a}\right)^2 + m \rhotil^{\gamma - 1}\left[-a^\gamma + \frac{\gamma}{\gamma -1}a^{\gamma - 1} - \frac{1}{\gamma -1}\right]=:f(a).
\end{align*}
Thus, it suffices to show that $f(a)\leq 0$ on $a \in [\tilde{a},\infty)$. To this end, observe that 
\begin{align*}
	f'(a) = \frac{m}{\rhotil^2 a^3}(a-1)\big(m^2 - \gamma a^{\gamma +1}\rhotil^{\gamma +1}\big).
\end{align*}
From \eqref{est:atil}, we deduce that $f'(a)$ has zeros at $a= \tilde{a}$ and $a=1$. Moreover, note that $f'(a)>0$ on $(\tilde{a}, 1)$, and $f'(a) < 0$ on $(1,\infty)$. Thus, $f(1)$ is the global maximum of $f$ on $[\tilde{a},\infty)$. Therefore, we conclude that 
\[
f(a) \leq f(1) = 0 \, \, \text{for all } a \in [\tilde{a},\infty).
\]
\end{proof}

\section{Modified front tracking algorithm}\label{sec:MFTA}
\setcounter{equation}{0}
In this section, we construct approximate solutions of the small BV solution. 
The construction follows the modified front tracking algorithm introduced by Amadori and Colombo in \cite{AC97}, which extends the original Bressan-Colombo scheme \cite{BC95} to the initial-boundary value problem. 
Besides the approximate solution constructed according to their procedure (referred to as the unshifted one), we also construct a shifted approximate solution whose shock travels with the speed $h$ introduced in Proposition \ref{Prop:S}.
Finally, we establish a contraction in a weighted $L^{1}$ distance between two unshifted approximate solutions, and an estimate in the same weighted $L^{1}$ distance between a shifted and an unshifted approximate solution.

The structure of this section closely parallels Section 4 of \cite{CKF25}, with fewer details. 
That section established the well-definedness of the shifted approximate solution 
and the $L^1$ estimate between a shifted and an unshifted approximate solution for the Cauchy problem. 
Here, we adapt their argument to the boundary setting. 
In particular, we verify that all key estimates in \cite{CKF25}, including the well-definedness of the shifted approximate solution and the $L^1$ estimate between shifted and unshifted approximate solutions, remain valid for the boundary value problem.

\subsection{The front tracking $\nu$-approximate solution}\label{sec:FTA}
In this section, we briefly review the front tracking algorithm as in Bressan and Colombo \cite{BC95}, and Amadori and Colombo \cite{AC97}. In this section, we use the Riemann coordinates $V = (V_1, V_2)$, where $V_1$ and $V_2$ are Riemann invariants defined in \eqref{def:RI}. 

For simplicity, with a slight abuse of notation, we denote by $R_i(V,\sigma)$ and $S_i(V,\sigma)$ the $i$-th rarefaction and shock curves in Riemann coordinates, from the state $V$. We note that, in Section~\ref{sec:bdset}, the analogous notation 
$S_i(U)$ was used for the Hugoniot loci written in terms of the conserved quantities $U$.

In the Riemann coordinates, the $i$-th rarefaction curves 
are given by
\[
R_1 (V,\sigma)  = (V_1 + \sigma, V_2), \quad R_2(V, \sigma)= (V_1, V_2 + \sigma).
\]
Since the $i$-th rarefaction curve and the $i$-th shock curve are tangent of order two at $V$, the $i$-th shock curve $S_i(V,\sigma)$ can be expressed as 
\[
S_1(V,\sigma) = (V_1 + \sigma, V_2 + \hat{S}_2(V,\sigma)\cdot \sigma^3), \quad S_2(V,\sigma) = (V_1 +\hat{S}_1(V,\sigma)\cdot \sigma^3, V_2 + \sigma ),
\]
for some smooth functions $\hat{S}_1$ and $\hat{S}_2$. Both $R_i$ and $S_i$ are parametrized by arc-length. 

In order to construct our front tracking algorithm, we first choose a smooth function $\phi:\mathbb{R} \to \mathbb{R}$ satisfying
\begin{equation}
\begin{aligned}\label{def of phi ft}
    \phi(s) = \begin{cases}
	1, \quad \text{if } s \leq -2,\\
	0, \quad \text{if } s \geq -1,
\end{cases}
\end{aligned}
\end{equation}
with $\phi'(s) \in [-2,0]$ for any $s \in [-2,-1]$.

For fixed $\nu>0$, we define \textit{the interpolated $i$-th wave curve} as follows:
\begin{align}\label{interpolated curve}
    \Phi_i^\nu(V,\sigma) := \phi\left(\frac{\sigma}{\sqrt{\nu}}\right)S_i(V,\sigma) + \left(1-\phi\left(\frac{\sigma}{\sqrt{\nu}}\right)\right)R_i(V,\sigma).
\end{align}
Note that $\Phi_i^\nu(V,\sigma) = S_i(V,\sigma)$ if $\sigma \leq -2\sqrt{\nu}$, and $\Phi_i^\nu(V,\sigma) = R_i(V,\sigma)$ if $\sigma \geq -\sqrt{\nu}$. In this section, if $\sigma >0$, we call $\sigma$ a jump of a rarefaction wave, and if $\sigma <0 $, a jump of a shock wave.
\begin{lem}
Let $V_b = V(U_b) \in \Omega_{sub}$ be the boundary data in the Riemann coordinates. There exists $\e_0 >0$ such that the following statements hold.
\begin{enumerate}
	\item For given $\nu, V^l, V^r$ with $V^l, V^r \in B_{\e_0}(V_b)$, 
    there exists a unique intermediate state $V^m$ and values $\sigma_1, \sigma_2$ such that 
	\[
	V^m = \Phi_1^\nu(V^l,\sigma_1), \quad V^r = \Phi_2^\nu(V^m,\sigma_2).
	\]
	\item For given $\nu, V^r$ and the inflow/outflow boundary condition (\ref{small BV BC}) with $V^r \in B_{\e_0}(V_b)$, there exists a unique interpolated Riemann solution to the the interpolated boundary Riemann problem.
\end{enumerate}
\end{lem}
\begin{proof}
The first statement is Lemma~1 of \cite{BC95}, 
and the proof of (2) can be found in Section 3 of \cite{AC97}.
\end{proof}
Now, we introduce the so-called $\nu$-approximate rarefaction and interpolated shock wave for the first family. Suppose that $\sigma >0$, and that $V^l=(V_1^l, V_2^l)$ and $V^r = (V_1^r, V_2^r)$ are connected by a 1-rarefaction, that is, $V^r = R_1(V_l,\sigma)$. Let $h, k \in \mathbb{Z}$ such that $h \nu \leq V_1^l < (h+1)\nu$ and $k \nu \leq V_1^r < (k+1)\nu$. Then, we define a $\nu$-approximate rarefaction fan by 
\[
V_\nu(x,t) := \begin{cases}
	V^l, \, \, &\text{if } x < \lambda_1(\hat{w}_1^h)t,\\
	w_1^j, \, \, &\text{if } \lambda_1(\hat{w}_1^{j-1})t < x < \lambda_1(\hat{w}_1^j)t,\\
	V^r, \, \, &\text{if } \lambda_1(\hat{w}_1^k)t < x,
\end{cases}
\quad j = h+1, \cdots, k,
\]
where $w_1^j := (j\nu,V_2^l)$ and $\hat{w}_1^j := ((j+1/2)\nu,V_2^l)$ for $j=h,\cdots, k$.

Assume that $\sigma < 0$, and that the states $V^l=(V_1^l, V_2^l)$ and $V^r = (V_1^r, V_2^r)$ are connected by an interpolated 1-shock curve, that is, $V^r = \Phi_1^\nu(V^l,\sigma)$. The interpolated 1-shock wave is then defined by
\begin{equation}\label{def:itpS}
	V_\nu(x,t) := \begin{cases}
	V^l \, \, &\text{if } x < \lambda_1^\phi(V^l,\sigma)t,\\
	V^r \, \, &\text{if } x > \lambda_1^\phi(V^l,\sigma)t,\\
\end{cases}
\end{equation}
where the interpolated shock speed is defined by 
\begin{align}\label{speed of shock unshifted}
    \lambda_1^\phi(V^l,\sigma) :=  \phi\left(\frac{\sigma}{\sqrt{\nu}}\right)\lambda_1^s(V^l,\sigma) + \left(1-\phi\left(\frac{\sigma}{\sqrt{\nu}}\right)\right)\lambda_1^r(V^l,\sigma).
\end{align}
Here, $\lambda_1^s(V^l,\sigma)$ and $\lambda_1^r(V^l,\sigma)$ are defined as 
\begin{equation*}
    \begin{aligned}
        &\lambda_1^s(V^l,\sigma) := s_1(V^l, S_1(V^l,\sigma)),\\
	    &\lambda_1^r(V^l,\sigma) := \sum_{j} \frac{\text{meas}([j\nu, (j+1)\nu] \cap [V_1^r, V_1^l])}{|\sigma|}\lambda_1(\hat{w}_1^j),
    \end{aligned}
\end{equation*}
where $s_1(V^l, S_1(V^l,\sigma))$ denotes the true 1-shock speed connecting $V^l$ to $S_1(V^l,\sigma)$.

Observe that if $\sigma \leq -2\sqrt{\nu}$, then \eqref{def:itpS} coincides with the true shock, as $\lambda_1^\phi(V^l,\sigma) = \lambda_1^s(V^l,\sigma)$.

Similarly, one can construct $\nu$-approximate waves for the second family. The detailed construction can be found in Section 2 of \cite{BC95}. We omit it here.

Let $V(x,0)$ be a piecewise constant initial data. We now construct the $\nu$-approximate solution $V_\nu(x,t)$ as follows.
\begin{enumerate}
	\item At $t=0$, we solve the interpolated Riemann problem at the interior discontinuity $(x>0)$, while at the boundary $(x=0)$, we solve the interpolated boundary Riemann problem. We keep the piecewise constant approximation constructed above until the first interaction or boundary hitting occurs, and denote this time by $t_1$.
	\item If wavefronts interact at $t=t_1$, we solve the corresponding interpolated Riemann problem again. If wavefronts hit the boundary at $t=t_1$, we solve the interpolated boundary Riemann problem.
	\item Repeating this process, we can define $V_\nu(x,t)$ for any finite time interval $[0,T]$.
\end{enumerate}
Proposition 1 of \cite{AC97} ensures the well-definedness of the approximate solution for all time, 
as well as the finiteness of the number of fronts and interactions on $[0,T]$ for any $T>0$.

\subsection{The shifted $\nu$-approximate solution}\label{sec:SFTA}
In order to establish the stability estimate for shock waves, we consider the shifted front tracking algorithm.

To construct the shifted $\nu$-approximate solution, if $\sigma \geq 0$, we use the same (approximate) rarefaction waves and wave speeds as in Section \ref{sec:FTA}.

Now, we consider the case when $\sigma <0$. At the interaction point $(x_1,t_1)$, if the interpolated Riemann problem generates a $i$-th shock $(V^l, \Phi_i^\nu(V^l,\sigma))$, the shifted $\nu$-approximate solution $\psi_\nu(x,t)$ is defined by
\begin{equation}\label{def:sFTS}
    \psi_\nu(x,t) := \begin{cases}
	V^l, \, \, &\text{if } h_L(t) < x < h(t),\\
	\Phi_i^\nu(V^l,\sigma), \, \, &\text{if } h(t) < x < h_R(t),
\end{cases}
\end{equation}
where $h(t)$ is the shift function given by Proposition~\ref{Prop:S} associated with the true shock $(V^l, S_i((V^l,\sigma)))$, and $h_L(t)$ and $h_R(t)$ denote the left and right fronts closest to $h(t)$, respectively.
\begin{rem}
    The shifted $\nu$-approximate solution is used to apply Proposition~\ref{Prop:S} at each shock discontinuity. Note that Proposition~\ref{Prop:S} provides the dissipation estimate only for true shocks. However, if $0>\sigma \geq -2 \sqrt{\nu}$, the interpolated shock waves in \eqref{def:sFTS} generally do not coincide with true shocks. Nevertheless, we retain this definition, as it guarantees that $V_*$ and $\psi$, defined in Lemma~\ref{lem: L1 difference}, form an elementary path essential for the proof of \eqref{l1 distance FT}. The error resulting from retaining this definition is quantified precisely in (6.11) and (6.12) of \cite{CKF25}.

    Although Proposition~\ref{Prop:S} requires true shocks, it can still be applied when $0>\sigma \geq -2 \sqrt{\nu}$, since the discrepancy between the true and interpolated shock waves can be quantified and shown to vanish as $\nu \to 0$ (see \eqref{main thm proof eq 3}).
\end{rem}

Since $\dot{h}$ satisfies (\ref{bound of h 1}) or (\ref{bound of h 2}), there exists a constant $c>0$ such that
\begin{align}\label{uniform bound on speed of the wave}
    \lambda^{(1)} \le -c < 0 < c \le \lambda^{(2)} \, ,
\end{align}
where $\lambda^{(i)}$ denotes any possible speed of an $i$-wave in the shifted and unshifted $\nu$-approximate solution. 
This property is crucial for ensuring the global well-definedness of the $\nu$-approximate solution.
Following the proof of Proposition~1 in Amadori and Colombo \cite{AC97}, which in turn relies on the finiteness argument of Section 4 in Bressan and Colombo~\cite{BC95}, 
we conclude that the shifted $\nu$-approximate solution is well-defined for all time and contains only finitely many wave fronts and interaction points on $[0,T]$ for any $T>0$. The proof proceeds in the same way as that for the unshifted $\nu$-approximate solution.

\subsection{Control on the BV norm}
To prove that the unshifted and shifted $\nu$-approximate solutions are globally well-defined, a standard approach 
is to show that the sum of the total variation $L$ and the interaction potential $Q$ 
(possibly multiplied by a sufficiently large constant) is nonincreasing in time. 
It is worth mentioning that these quantities need to be slightly modified because of the presence 
of a boundary. In particular, (6.2) of Amadori and Colombo~\cite{AC97} provides 
the following definitions of $L$ and $Q$ adapted to the boundary setting:
\begin{align}\label{def of V, Q}
    L = \sum_{i,\alpha} K_i |\sigma_{i,\alpha}| \, , \qquad 
    Q = \sum_{(\sigma_{i,\alpha},\sigma_{j,\beta}) \in \mathcal{A}} 
        K_i K_j |\sigma_{i,\alpha}\sigma_{j,\beta}| \, ,
\end{align}
where $K_2=1$, $K_1\geq 1$ is a sufficiently large constant, and $\mathcal{A}$ denotes the usual set of pairs of approaching waves, i.e., those pairs located at $x_\alpha<x_\beta$ such that either $i=2,\, j=1$ or $i=j$ and at least one of them is a shock. 

When a 1-wave $\sigma_1$ hits the boundary and generates a reflected $2$-wave $\sigma_2$ at time $t$. It is easy to see that $\sigma_1$ and $\sigma_2$ have the same sign, i.e., either both shocks or both rarefactions. Moreover, there exists a constant 
$\lct>1$ valid for all $U\in B_{\e_0}(U_b)$ such that 
\begin{align}\label{def of lct}
    |\sigma_2| \leq \lct |\sigma_1| \, .
\end{align}
In this paper, we choose $K_1$ 
\begin{align}\label{relation K_1, C_1}
    K_1 = 4 \lct \, .
\end{align}
Lemma 3 in Amadori and Colomobo \cite{AC97} shows that for sufficiently large $\kappa>0$, if the total variation at time $0$ is sufficiently small, then for any $t\geq 0$, 
\begin{align}\label{nonincreasing V+Q}
    \Delta (L+\kappa Q)(t) \leq 0 \, ,
\end{align}
and in particular, for any $t_*\geq 0$ at which an interaction occurs,
\begin{align}\label{nonincreasing V+Q interaction}
    \Delta (L+\kappa Q)(t_*) \leq \Delta Q(t_*) \, .
\end{align}
For any $t_*\geq 0$ at which the waves in a set $\mathcal{S}_1$ of $1$-waves hit the boundary and a set $\mathcal{S}_2$ of $2$-waves is generated,
\begin{align*}
    \Delta (L+\kappa Q)(t) 
    \leq -K_1\sum_{\mathcal{S}_1}|\sigma_1| + \sum_{\mathcal{S}_2}|\sigma_2| + \kappa L(t^-) \sum_{\mathcal{S}_2} |\sigma_2| \leq -\frac{K_1}{2} \sum_{\mathcal{S}_1}|\sigma_1| \, .
\end{align*}

\subsection{$L^1$ estimates between $\nu$-approximate solutions}
To measure the distance between $\nu$-approximate solutions, 
we employ the weighted distance functional $d_\nu$ introduced by 
Bressan and Colombo \cite{BC95} and later adapted to the boundary setting by Amadori and Colombo \cite{AC97}. 
The functional $d_\nu$ quantifies the total weighted strength of the pseudopolygonal paths connecting two piecewise constant approximate solutions corresponding to the same far-field and boundary conditions. Moreover, $d_\nu$ defines a metric equivalent to the $L^1$ distance.

We recall the essential property that $d_\nu$ is nonincreasing when two unshifted $\nu$-approximate solutions evolve, and we refer to \cite{BC95} and \cite{AC97} for the proof of the contraction property. For completeness, we will briefly present the main results in \cite{AC97} in what follows.

For a fixed constant state $d$, we define a domain of small BV space as follows:
\[
\mathcal{D} := \{V - d \in L^1(\Rp) \, :\, V \, \text{is a piecewise constant, } L(V) + \kappa Q(V) < \delta\},
\]
where $\kappa$ is defined in \eqref{nonincreasing V+Q}. Observe that, for sufficiently small $\delta$, (un)shifted $\nu$-approximate solutions belong to $\mathcal{D}$ for a.e. $t\geq 0$.

In the sequel, we provide definitions of elementary path and pseudopolygonals.
\begin{defi}
An elementary path $\gamma:
 (a,b) \rightarrow \mathcal{D}$ is a map of the form
\begin{equation}\label{def:EP}
    \gamma(\theta) = \omega^{0}\cdot \chi_{(0,x_1^\th]} + \sum_{\alpha=1}^{n-1}\omega^\alpha \cdot \chi_{(x_\alpha^{\th}, x_{\alpha+1}^\th]} + d\cdot \chi_{(x_n^\th, +\infty)}, \quad x_\alpha^\th = x_\alpha + \xi_\alpha \theta,
\end{equation}
with $x_\alpha^\th < x_{\alpha + 1}^\th$ for all $\th \in (a,b)$ and $\alpha=1,2,\cdots, n-1$.
\end{defi}
\begin{defi}\label{def:PP}
A continuous map $\gamma:[a,b] \to L^1_{loc}$ is called pseudopolygonal if there exists a countable family of disjoint open intervals $J_h \in [a,b]$ such that $[a,b]\setminus \cup_{h \in \mathbb{N}} J_h$ is countable, and the restriction of $\gamma$ to each $J_h$ is an elementary path.
\end{defi}
Loosly speaking, the pseudopolygonal is the countable concatenations of the elementary paths. Moreover, we remark that any two piecewise constant functions can be joined by pseudopolygonals.

As shown in \cite{AC97}, if we denote $\mathcal{S}_t$ the semigroup generated by the front tracking scheme in Section~\ref{sec:FTA}, and $\gamma_0(\theta)$ is a pseudopolygonal, then $$\gamma_t(\theta):=\mathcal{S}_t\circ\gamma_0(\theta)$$ is also a pseudopolygonal. More precisely, there exists a countable family of intervals $\{J_h\}_{h \in \mathbb{N}}$ such that $[a,b]\setminus \cup_{h \in \mathbb{N}} J_h$ is countable, and the wave configuration of $\nu$-approximate solutions with data $\gamma_0(\theta)$ remains constant as $\theta$ varies within each interval $J_h$.

Let $\gamma$ be an elementary path in \eqref{def:EP}. Define
\begin{equation}\label{def:EPnorm}
    \|\gamma\| := (b-a)\Upsilon_\x,
\end{equation}
where 
\begin{align*}
    &\Upsilon_\x := \left(\sum_{\alpha=1}^2 \sum_{i=1}^2 H_i(2+\sgn \sigma_{i,\alpha})|\sigma_{i,\alpha}\xi_{\alpha}|+H_3 \mathcal{Q}_\xi\right)e^{H_4 (L+Q)},\\
    &\mathcal{Q}_\xi = \sum_{(\sigma_{i,\alpha},\sigma_{j,\beta})\in\mathcal{A}}H_i H_j |\sigma_{i,\alpha}\sigma_{j,\beta}|\big((2+\sgn \sigma_{i,\alpha})|\xi_{i,\alpha}|+(2+\sgn \sigma_{j,\beta})|\xi_{j,\beta}|\big),
\end{align*}
where $H_1, H_2, H_3$, and $H_4$ are suitable positive constants. Here, $\sigma_{1,\alpha}$ and $\sigma_{2,\alpha}$ denote the sizes of the 1- and 2-waves in the interpolated Riemann solution corresponding to the data $\omega^{\alpha}$ and $\omega^{\alpha+1}$, respectively, and $\xi_{i,\alpha}$ denotes the shift speed of the wave $\sigma_{i,\alpha}$, as defined in \eqref{def:EP}.

Now, let $\gamma$ be a pseudopolygonal with the family of intervals $\{J_h\}$ as in Definition~\ref{def:PP}, and denote by $\gamma_h$ the restriction of $\gamma$ on $J_h$. Define
\[
\|\gamma\| := \sum_{h \in \mathbb{N}}\|\gamma_h\|.
\]
We then define the weighted distance as 
\begin{equation}\label{def:dnu}
    d_\nu(V,W):= \inf\{\|\gamma\|\,:\, \gamma:[a,b] \to \mathcal{D} \, \, \text{is a pseudopolygonal joining } V \, \, \text{with } W \}.
\end{equation}
\begin{prop}\label{prop L^1 semigroup}\cite{AC97}
There exist positive constants $K_1,K_2, H_1, H_2, H_3$ and $H_4$ such that the following holds.

Let $V$ and $W$ be $\nu$-approximate solutions constructed in 
Section \ref{sec:FTA} or Section \ref{sec:SFTA}. 
The weighted distance $d_\nu$ is equivalent to the $L^1$ distance; i.e.,
\[
K_1 \|V-W\|_{L^1(\Rp)} \leq d_\nu(V,W) \leq K_2 \|V-W\|_{L^1(\Rp)}.
\]
In particular, if both $V$ and $W$ are unshifted $\nu$-approximate solutions, 
then $d_\nu$ is contractive; that is,
\begin{align*}
    d_\nu(V(\cdot,t), W(\cdot,t)) 
    \leq { \, d_\nu(V(\cdot,0), W(\cdot,0))} \, , \; \textup{ for any } t \geq 0 \, .
\end{align*}
\end{prop}

We also obtain an $L^1$ estimate on the difference between a shifted and an unshifted $\nu$-approximate solution by following the argument of Lemma 4.2 in \cite{CKF25}. 
At each time an interaction or a boundary hit occurs, the variation of the distance between the shifted and unshifted approximate solutions is determined by the shock strengths and the difference between the shifted and interpolated wave speeds associated with the shocks.
The overall proof follows the same reasoning as in \cite{CKF25}, except that boundary hits must also be taken into account; however, since the uniform bounds on the wave speeds in (\ref{uniform bound on speed of the wave}) hold for all fronts in both the shifted and unshifted $\nu$-approximate solutions, the same argument applies.

\begin{lem}\label{lem: L1 difference}
Let $T>0$.
Let $V$ be an unshifted $\nu$-approximate solution constructed on $[0,T]$ as described in Section \ref{sec:FTA} and $\psi$ a shifted $\nu$-approximate solution constructed on $[0,T]$ as described in Section \ref{sec:SFTA}. Then in the sense of distributions,
\begin{equation}\label{l1 distance FT}
   \frac{d}{dt}d_\nu(V(\cdot, t), \psi(\cdot, t)) \leq C \sum_{\alpha \in \mathcal{S}(t)} |\sigma_\alpha|\,|\dot{h}_\alpha(t) - \dot{h}_{true, \alpha}(t)|\,d\tau \, ,
\end{equation}
where $\mathcal{S}(t)$ is the set of all shocks in $\psi$ at time $t$, $\sigma_\alpha$ is the strength of the shock $\alpha$, $\dot{h}_\alpha$ is the shifted speed defined for the shock $\alpha$ in Proposition \ref{Prop:S}, and $\dot{h}_{true,\alpha}$ is the speed of the interpolated shock $\alpha$ starting from the jump discontinuity $x_\alpha(t)$ in $\psi$, as constructed in Section \ref{sec:FTA}.
\end{lem}
\begin{figure}[h!]
\centerline{\includegraphics[width=0.50\textwidth]{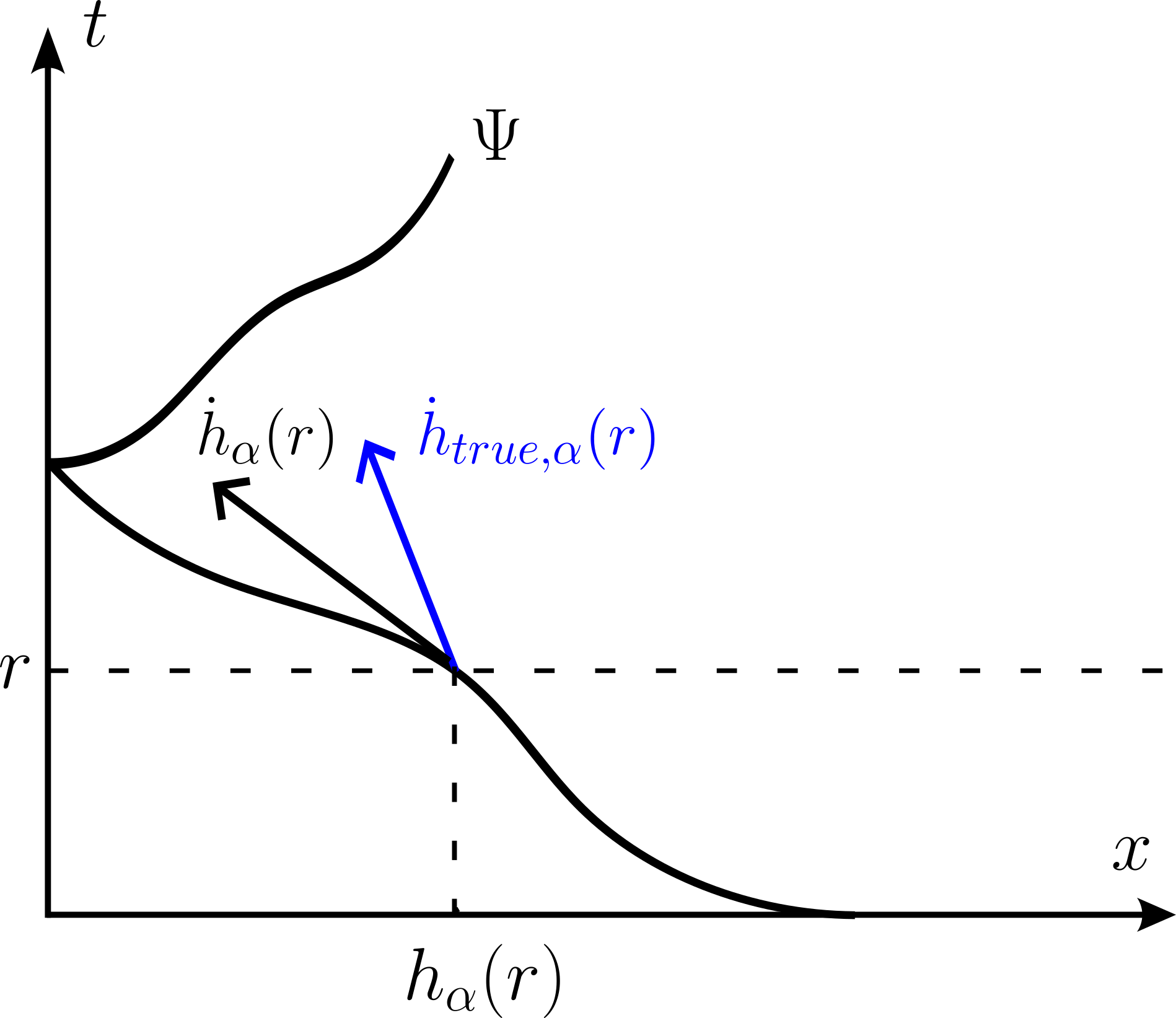}}
\caption[$\dot{h}_{\alpha}$ and $\dot{h}_{true,\alpha}$]{$\dot{h}_\alpha$ denotes the shifted speed of the shock $\alpha$ as defined in Proposition~\ref{Prop:S}, while $\dot{h}_{\mathrm{true}, \alpha}$ represents the true shock speed connecting the left and right states $\psi(r,h_\alpha(r)-)$ and $\psi(r,h_\alpha(r)+)$, as constructed in Section \ref{sec:FTA}.} \label{fig:htrue}
\end{figure}
We refer the reader to Figure~\ref{fig:htrue} for an illustration of $\dot{h}_{\mathrm{true},\alpha}$ and $\dot{h}_{\alpha}$.

\begin{proof}
Although the proof is similar to that in \cite{CKF25}, we provide a proof here for completeness.

\noindent
\textbf{Same wave configuration.}
Let $0\leq \ts< T$. We define $V_*(x,t)$ locally as the unshifted $\nu$-approximate solution constructed both backward and forward in time from the data $\psi(x,\ts)$. We claim that there exists $\delta>0$ such that $V_*$ is well-defined on $(\ts-\delta,\ts+\delta)\cap[0,T]$ and has the same wave configuration as $\psi$; that is, $V_*$ and $\psi$ consist of the same waves with identical relative positions.

Assume first that $\psi$ has neither an interaction nor a boundary interaction at time $\ts$. By the uniform bound on the speeds of all waves in \eqref{uniform bound on speed of the wave}, there exists $\delta>0$ such that neither $V_*$ nor $\psi$ undergoes an interaction or boundary interaction on $(\ts-\delta,\ts+\delta)\cap[0,T]$. Hence, they have the same wave configuration throughout this interval.

Assume now that an interaction or boundary interaction occurs at time $\ts$. We first observe that, in both the shifted and unshifted front-tracking algorithms, the speeds of rarefaction fronts are prescribed in the same way, and \eqref{uniform bound on speed of the wave} ensures that every front of the first family has negative speed, whereas every front of the second family has positive speed. On the other hand, consider two adjacent waves in the unshifted approximate solution, so that the right state of the left wave coincides with the left state of the right wave. If at least one of them is a shock, then the Lax entropy condition implies that the left wave travels faster than the right wave. Together with the fact that two rarefaction fronts of the same family cannot approach each other, these observations imply that the unshifted $\nu$-approximate solution, whether constructed locally forward or backward in time, has the same wave configuration as $\psi$, although the corresponding waves may propagate at different speeds.

\noindent
\textbf{Quantitative estimates.} For any $\ts-\delta \leq t_1 \leq \ts \leq t_2\leq \ts + \delta$, Proposition \ref{prop L^1 semigroup} implies that
\begin{equation*}
    \begin{aligned}
        &d_\nu(V(t_2,\cdot),\psi(t_2,\cdot))-d_\nu(V(t_1,\cdot),\psi(t_1,\cdot))\\
        &\leq d_\nu(V(t_2,\cdot),V_*(t_2,\cdot)) + d_\nu(V_*(t_2,\cdot),\psi(t_2,\cdot))\\
        &\quad  - d_\nu(V(t_1,\cdot),V_*(t_1,\cdot)) +
        d_\nu(V_*(t_1,\cdot),\psi(t_1,\cdot)) \\
        &\leq d_\nu(V_*(t_2,\cdot),\psi(t_2,\cdot)) + d_\nu(V_*(t_1,\cdot),\psi(t_1,\cdot)) \rightarrow 0  \, .
    \end{aligned}
\end{equation*}
Letting $t_1$ and $t_2$ tend to $t_*$, we obtain
\begin{equation}\label{l1 distance FT 2}
   \limsup_{s \to t_*+} \; d_\nu(V(\cdot,s),\psi(\cdot,s)) \leq \liminf_{s \to t_*-}\; d_\nu(V(\cdot,s),\psi(\cdot,s)) \, .
\end{equation}
On the other hand, Proposition \ref{prop L^1 semigroup} implies that for $\ts\leq t\leq \ts+\delta$,
\begin{equation}\label{est:l1}
    \begin{aligned}
        &d_\nu(V(t,\cdot),\psi(t,\cdot))-d_\nu(V(\ts+,\cdot),\psi(\ts+,\cdot)) \\
        &= d_\nu(V(t,\cdot),\psi(t,\cdot))-d_\nu(V(\ts+,\cdot),V_*(\ts+,\cdot)) \\
        &\leq d_\nu(V(t,\cdot),\psi(t,\cdot))-d_\nu(V(t,\cdot),V_*(t,\cdot)) \\
        &\leq d_\nu(\psi(t,\cdot),V_*(t,\cdot)).
    \end{aligned}
\end{equation}
We define an elementary path, denoted by $\gamma_*$, joining $\psi(t,\cdot)$ to $V_*(t,\cdot)$ as
\[
x_\alpha^\th(s) = x^*_\alpha(s) + \xi_\alpha \theta, \quad \theta \in [0,1], \, \,  s \in [t_*,t].
\]
Here, $x^*(s)$ for $[t_*,t]$ and $x_\alpha^\theta(t)\big|_{\theta=1}$ are the positions of any pair of shock waves, for $V_*$ and $\psi$, respectively. Since the locations of rarefaction fronts for $\psi$ and $V_*$ are the same, it follows from \eqref{def:EPnorm}, \eqref{def:dnu}, and \eqref{nonincreasing V+Q}, that 
\begin{equation}\label{l1 distance eq 1}
    \begin{aligned}
        d_\nu(\psi(t,\cdot),V_*(t,\cdot)) &\leq \|\gamma_*(t)\| \leq C \sum_{\alpha\in \mathcal{S}(t)} |\sigma_\alpha \xi_\alpha|\\
        &\leq C \sum_{\alpha\in \mathcal{S}(t)} |\sigma_\alpha| |x_\alpha(t) - x^*_\alpha(t)|\\
        &\leq C \sum_{\alpha\in \mathcal{S}(t)} |\sigma_\alpha| \int_{\ts}^t \,|\dot{h}_\alpha(s) - \dot{h}_{true, \alpha}(s)| \, ds \, .
    \end{aligned}
\end{equation}
By \eqref{l1 distance FT 2} and \eqref{l1 distance eq 1}, we complete the proof.

\end{proof}

\section{Boundary condition for the limit of the unshifted approximate solutions}
\label{sec combination and separation}
The goal of this section is to prove Proposition \ref{lem:BC}, which asserts that since unshifted $\nu$-approximate solutions satisfy the inflow/outflow boundary condition \eqref{small BV BC}, any limit of them (up to a subsequence) also satisfies the boundary condition \eqref{small BV BC}. 
This result will later be used in Section \ref{subsec proof of them main theorem} to identify the limit (up to a subsequence) of the unshifted $\nu$-approximate solutions as the small BV solution, thereby establishing the existence and uniqueness of small BV solutions.

Over time, shocks of the same family may merge, whereas rarefactions may be partitioned. In Section~\ref{subsec label}, we associate with each wave a set of elements encoding its formation history and assign a tracked strength to each element. These tracked strengths are defined so that their sum over the elements in the set associated with a wave is exactly equal to the strength of that wave. In Section~\ref{subsec evolution}, we prove that the tracked strength of each individual element is bounded by its strength at the time when it is introduced. Consequently, after fixing $x=x_*$, we can control the total variation of the waves crossing the line $x=x_*$. This is the main idea underlying the proof of Proposition~\ref{lem:BC}, which is given in Section~\ref{subsec limit of approximate solutions}.

Now, fix $T>0$. By slightly perturbing the initial data of the unshifted $\nu$-approximate solutions, we may assume that the unshifted $\nu$-approximate solutions satisfy the following properties:
\begin{itemize}
  \item[1)] no multiple-wave interaction or multiple-wave hitting on $[0,T]$: each interaction involves exactly two incoming waves, and each boundary hitting involves only one incoming wave.
  \item[2)] at most one event at a time on $[0,T]$: at each time $t$, there is at most one event (either an interaction or a boundary hitting).
\end{itemize}

\subsection{Wave-associated sets, elements, and tracked strengths} \label{subsec label}

Since shocks may merge forward in time, whereas rarefactions may be partitioned forward in time, we define the shock-associated sets forward in time and the rarefaction-associated sets backward in time, as follows.

At time $t=0$, we enumerate the shocks from left to right and assign to them the sets
\begin{align*}
    \{-1\},\ \ldots,\ \{-N_S(0)\}.
\end{align*}
Here, $N_S(t)$ denotes the total number of elements introduced into the sets associated with shocks. It will be defined forward in time and changes only at an event. In particular, $N_S(0)$ is the number of shocks at time $t=0$.

For rarefactions, we proceed backward in time. At time $T+$, we enumerate the rarefactions from left to right and assign to them the sets
\begin{align*}
    \{1\},\ \ldots,\ \{N_R(T+)\}.
\end{align*}
Here, $N_R(t)$ denotes the total number of elements introduced into the sets associated with rarefactions. It will be defined backward in time and changes only at an event. In particular, $N_R(T+)$ is the number of rarefactions at time $t=T+$.

To each element $b$ in a set associated with either a shock or a rarefaction, we assign a tracked strength $\s_b$. These tracked strengths may change only at an event. At $t=0$, for each shock element $b$, we set $\s_b(0)$ equal to the strength of the corresponding shock, whose associated set is $\{b\}$. Similarly, at $t=T+$, for each rarefaction element $b$, we set $\s_b(T+)$ equal to the strength of the corresponding rarefaction, whose associated set is $\{b\}$.

Before describing how sets are assigned to waves and how new elements are introduced at an event, either an interaction or a boundary hitting, we recall the following wave estimate proved in \cite{BC95}.

\begin{lem}\cite{BC95}\label{lem:Glimm}
Assume two waves $\s, \widetilde{\s}$ interact at time $t$. Let $\s'$, $\s''$ be the strengths of the outgoing waves.

Then, there exists a constant $C_0$ (uniformly valid for $U \in B_{\e_0}(U_b)$) such that the followings hold.\\
$\bullet$ If $\s$, $\widetilde{\s}$, and $\s'$ are $i$-waves, and $\s''$ is a $j$-wave with $i\neq j$, then
\[
|\s' - (\s + \widetilde{\s})| + |\s''| \leq C_0 |\s \widetilde{\s}| (|\s| + |\widetilde{\s}|).
\]
$\bullet$ If $\s$ and $\s'$ are $1$-waves, and $\widetilde{\s}$ and $\s''$ are $2$-waves, then 
\[
|\s - \s'| + |\widetilde{\s} - \s''| \leq C_0 |\s \widetilde{\s}| (|\s| + |\widetilde{\s}|).
\]
\end{lem}

\subsubsection{When an event occurs}\label{def of separation interactions}

Assume that, at time $t_*$, either two waves interact or a $1$-wave hits the boundary, and that one of the outgoing waves is an $i$-shock of strength $\s'$. By Lemma~\ref{lem:Glimm}, one of the following four cases occurs: two incoming $i$-shocks interact; exactly one incoming wave in an interaction is an $i$-shock; $i=2$ and a $1$-shock hits the boundary; or no incoming $i$-shock is involved, in which case the event is an interaction between two waves of the other family.

In the first case, we denote the strengths of the two incoming $i$-shocks by $\s$ and $\widetilde{\s}$, respectively. The sets associated with the two incoming $i$-shocks are of the form
\begin{align*}
    \{b_1,\ldots,b_n\} \, , \; \{b_{n+1},\ldots,b_m\} \, ,
\end{align*}
where $n,m \in \mathbb{N}_*$, $1\leq n<m$, and $-b_l\in\mathbb{N}_+$ for $1\leq l\leq m$.
Since the two incoming $i$-shocks merge into the outgoing $i$-shock, we assign the set
\begin{align*}
    \{b_1,\ldots,b_m\}
\end{align*}
to the outgoing $i$-shock (See Figure~\ref{Scase1}). For each $1\leq l\leq m$, we define the post-interaction tracked strength by
\begin{align*}
    \s_{b_l}(t_*+) = \frac{\s_{b_l}(t_*-)}{\s+\widetilde{\s}}\s' \, .
\end{align*}

From this definition, each element's strength is tracked exactly through the interaction. Indeed, the post-interaction tracked strengths are chosen to be proportional to the corresponding pre-interaction tracked strengths. More precisely, they preserve exactly the relative proportions of the $m$ tracked strengths before the interaction, while their sum is equal to the strength $\s'$ of the outgoing wave. That is,
\[
\frac{\s_{b_l}(t_*+)}{\s_{b_l}(t_*-)}=\frac{\s'}{\s+\widetilde{\s}},\quad \sum_{l=1}^m \s_{b_l}(t_*+) = \s', \quad \sum_{l=1}^n \s_{b_l}(t_*-) = \s,\quad \sum_{l=n+1}^m \s_{b_l}(t_*-) = \widetilde{\s}.
\]

In the second and third cases, we denote by $\s$ the strength of the incoming shock that gives rise to the outgoing $i$-shock: the incoming $i$-shock in the second case and, in the third case, where $i=2$, the incoming $1$-shock (See Figure~\ref{Scase2} and Figure~\ref{Scase3}). The sets associated with this incoming shock is of the form
\begin{align*}
    \{b_1,\ldots,b_n\} \, , 
\end{align*}
where $n \in \mathbb{N}_*$ and $-b_l\in\mathbb{N}_+$ for $1\leq l\leq n$.
We assign the same set to the outgoing $i$-shock. For each $1\leq l\leq n$, we define the post-interaction tracked strength by
\begin{align*}
    \s_{b_l}(t_*+) = \frac{\s_{b_l}(t_*-)}{\s}\s' \, .
\end{align*}

In the fourth case, both incoming waves belong to the other family (See Figure~\ref{Scase4}). We assign the set
\begin{align*}
    \{-N_S(t_*-)-1\}
\end{align*}
to the outgoing $i$-shock, and define 
\begin{align*}
    \s_{-N_S(t_*-)-1}(t_*+) = \s ' \, ,
\end{align*}
We then update
\begin{align*}
    N_S(t_*+) = N_S(t_*-) + 1 \, .
\end{align*}

\begin{figure}[!htbp]
\centering
\includegraphics[width=0.4\textwidth]{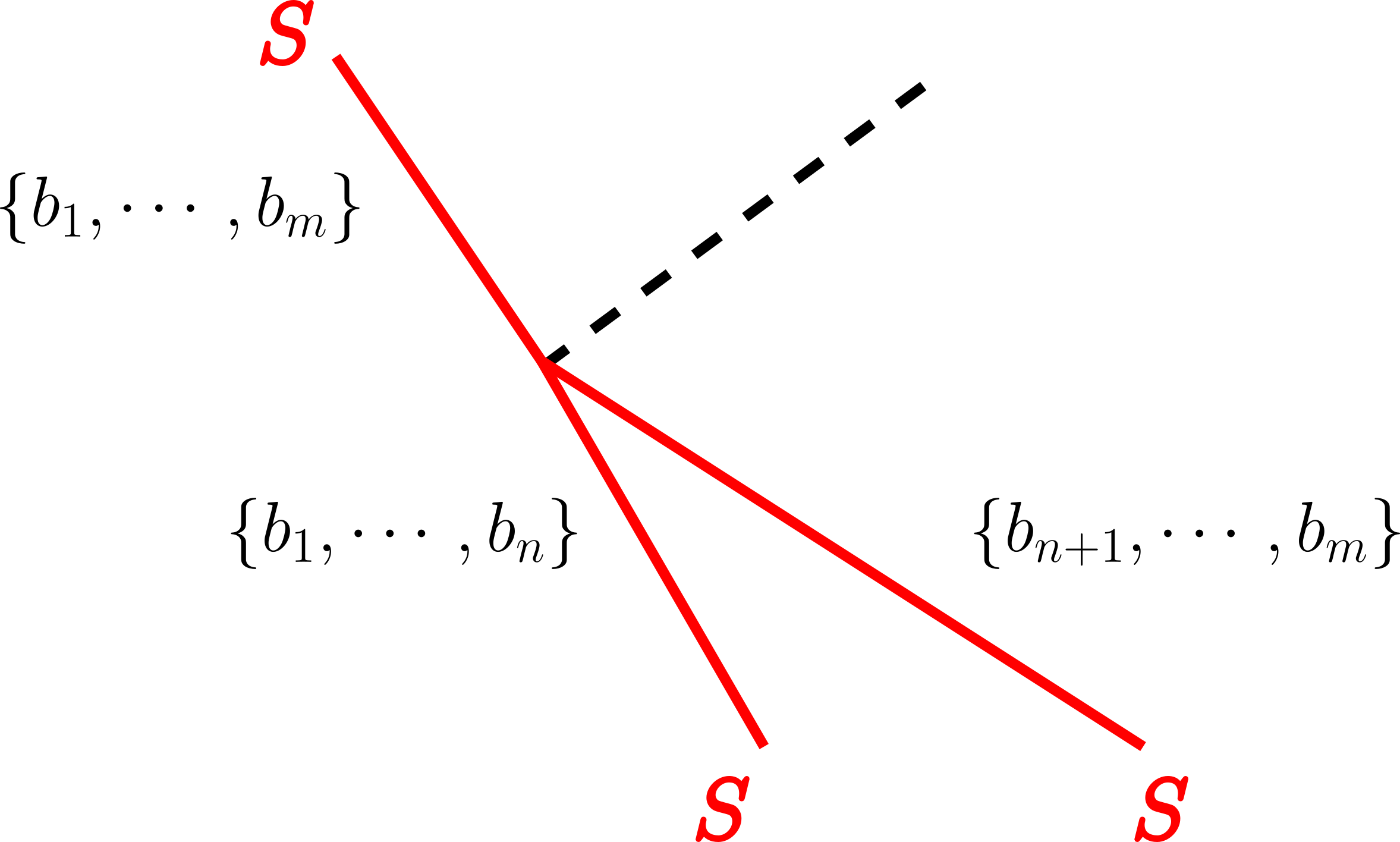}
\caption{Two incoming $i$-shocks merge into the outgoing $i$-shock.}
\label{Scase1}
\end{figure}

\begin{figure}[!htbp]
\centering
\begin{subfigure}{0.4\textwidth}
    \centering
    \includegraphics[width=\linewidth]{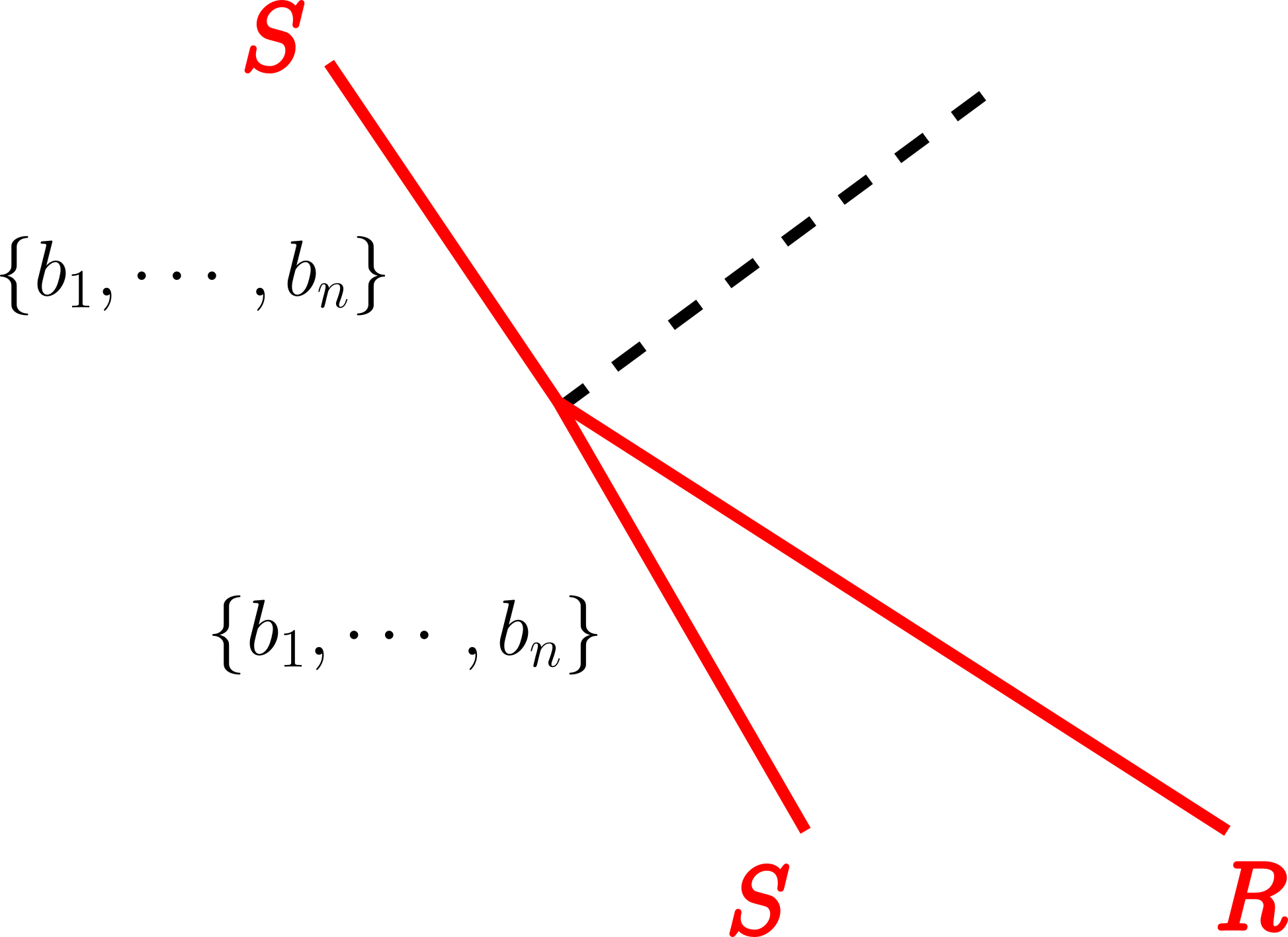}
\end{subfigure}
\hfill
\begin{subfigure}{0.4\textwidth}
    \centering
    \includegraphics[width=\linewidth]{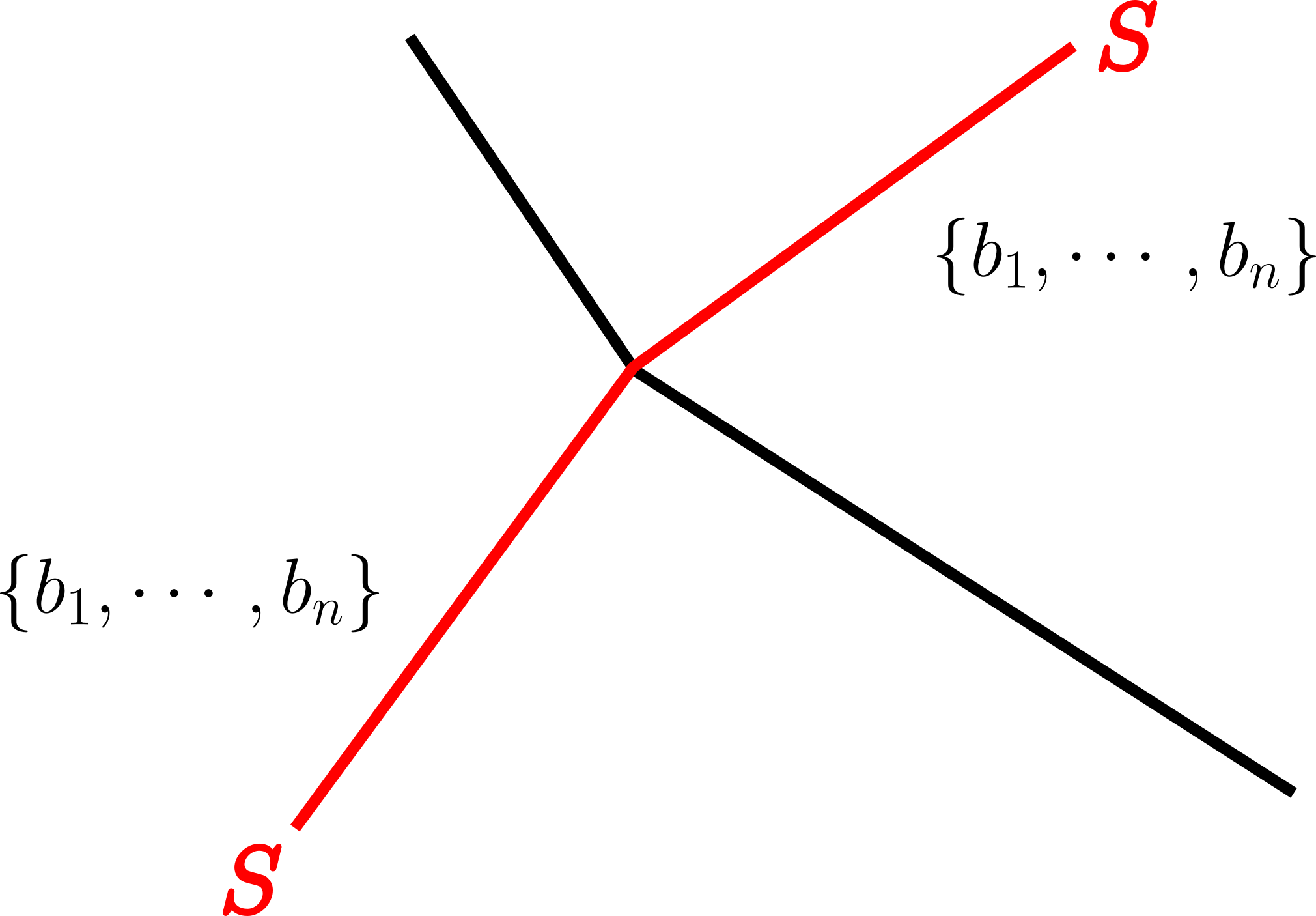}
\end{subfigure}
\caption{Exactly one of the incoming waves and the outgoing wave are $i$-shocks.}
\label{Scase2}
\end{figure}

\begin{figure}[!htbp]
\centering
\includegraphics[width=0.25\textwidth]{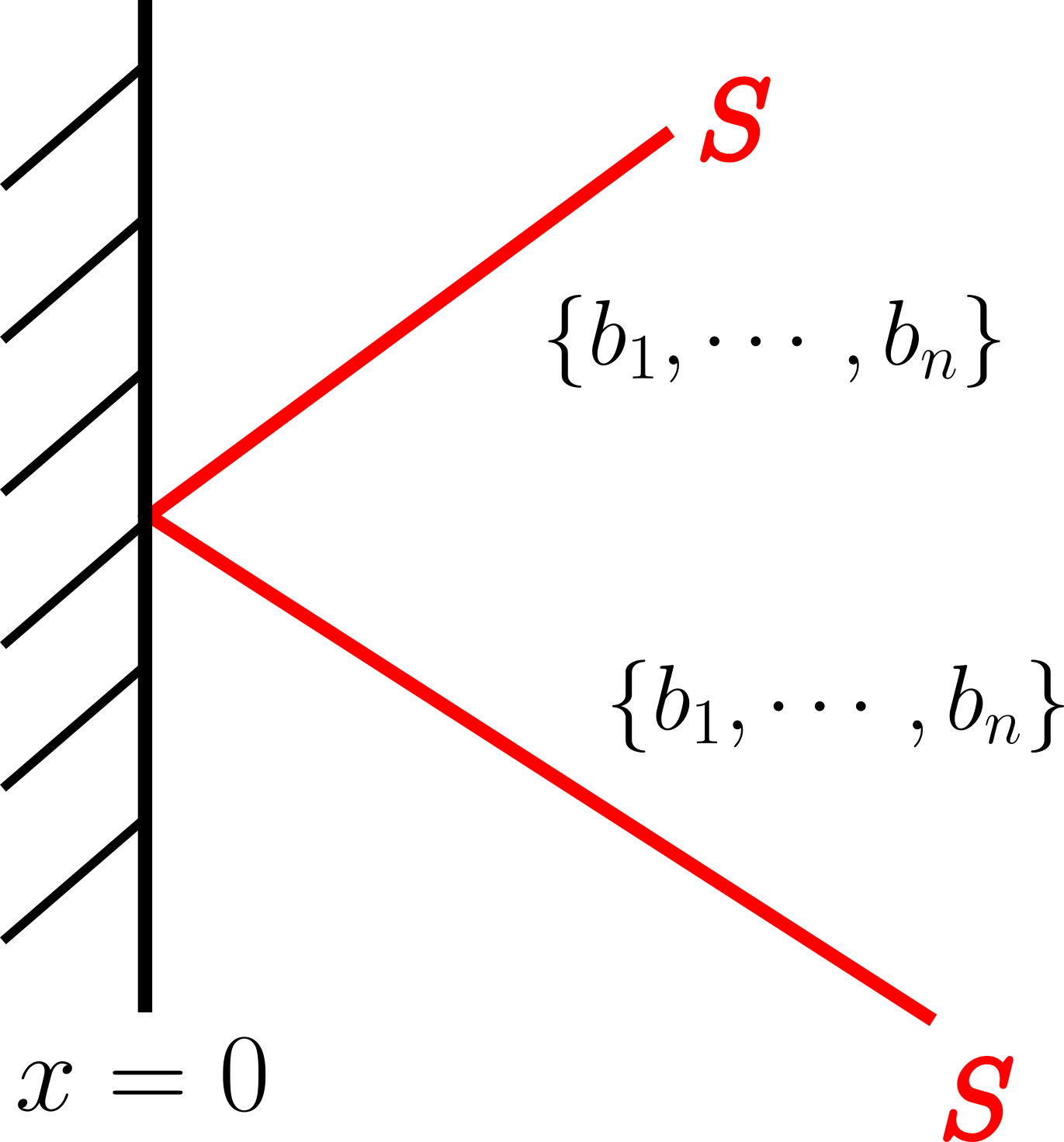}
\caption{$1$-shock hits the boundary.}
\label{Scase3}
\end{figure}

\begin{figure}[!htbp]
\centering
\includegraphics[width=0.3\textwidth]{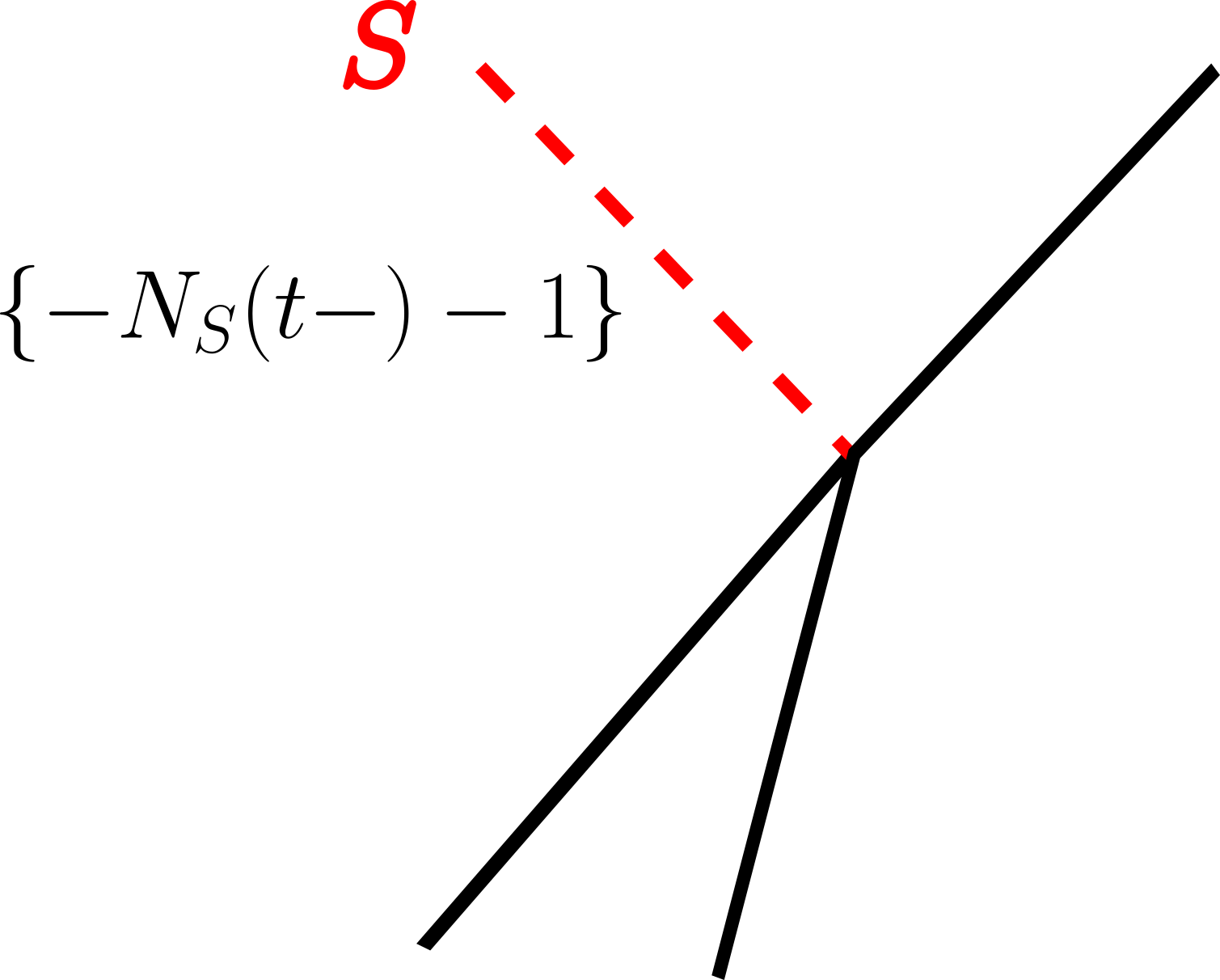}
\caption{No incoming $i$-shock is involved, whereas the outgoing wave is an $i$-shock.}
\label{Scase4}
\end{figure}
\medskip

Assume that, at time $t_*$, either two waves interact or a $1$-wave hits the boundary, and that one of the incoming waves is an $i$-rarefaction of strength $\s$. By Lemma~\ref{lem:Glimm}, one of the following three cases occurs: an interaction occurs and the outgoing waves contain exactly $M\in\mathbb{N}_+$ $i$-rarefactions; $i=1$ and a $1$-rarefaction hits the boundary; or an interaction occurs and none of the outgoing waves is an $i$-rarefaction.

In the first and second cases, we denote by $\s'$ the total strength of the $M$ outgoing rarefactions generated by the incoming $i$-rarefaction: the outgoing $i$-rarefactions in the first case and, in the second case where $i=1$, the outgoing $2$-rarefactions (See Figure~\ref{Rcase1} and Figure~\ref{Rcase2}). The sets associated with these outgoing rarefactions are of the form
\begin{align*}
    \{b_1^{(1)},\ldots,b_{n_1}^{(1)}\} \, , \; ... \, , \; 
    \{b_1^{(M)},\ldots,b_{n_M}^{(M)}\} \, ,
\end{align*}
where $n_k\in\mathbb{N}_*$ for $1\leq k \leq M$, and $b_l^{(k)}\in\mathbb{N}_*$ for $1\leq k \leq M$ and $1\leq l\leq n_k$. Since the incoming $i$-rarefaction is partitioned into these $M$ outgoing rarefactions, we assign to the incoming $i$-rarefaction the set
\begin{align*}
    \{b_1^{(1)},\ldots,b_{n_1}^{(1)} \, , \; \ldots \, , \; b_1^{(M)},\ldots,b_{n_M}^{(M)}\}.
\end{align*}
For each $1\leq k\leq M$ and $1\leq l\leq n_k$, we define the pre-interaction tracked strength by
\begin{align*}
    \s_{b_l^{(k)}}(t_*-) = \frac{\s_{b_l^{(k)}}(t_*+)}{\s'}\s \, .
\end{align*}

In the third case, none of the outgoing waves is an $i$-rarefaction (See Figure~\ref{Rcase3}). We assign the set
\begin{align*}
    \{N_R(t_*+)+1\}
\end{align*}
to the incoming $i$-rarefaction, and define 
\begin{align*}
    \s_{N_R(t_*+)+1}(t_*-) = \s \, ,
\end{align*}
We then update
\begin{align*}
    N_R(t_*-) = N_R(t_*+) + 1 \, .
\end{align*}

\begin{rem}
We define the location of an element to be the location of the wave whose associated set contains that element. Recall Section \ref{def of separation interactions}. At an interaction, if an element belongs to the set associated with an incoming $i$-wave and to the set associated with an outgoing $j$-wave, then necessarily $i=j$. 

Hence, an element generated in a set associated with a $2$-wave remains in sets associated with $2$-waves and travels to the right until it disappears. Conversely, an element generated in a set associated with a $1$-wave remains in sets associated with $1$-waves and travels to the left until it either disappears or the $1$-wave whose set contains it hits the boundary. In the latter case, the element becomes associated with a $2$-wave and subsequently travels to the right until it disappears.

We say that two elements interact if the waves whose associated sets contain them interact. The preceding observation yields the following result.
\end{rem}
\begin{lem}\label{lemma property of elements}
    For the elements defined in Section \ref{subsec label},
    \begin{itemize}
        \item Any two elements can interact at most twice.
        \item For every fixed $x_*$, each element crosses the line $x=x_*$ at most twice over its lifetime.
    \end{itemize}
\end{lem}

\begin{figure}[!htbp]
\centering
\begin{subfigure}{0.45\textwidth}
    \centering
    \includegraphics[width=\linewidth]{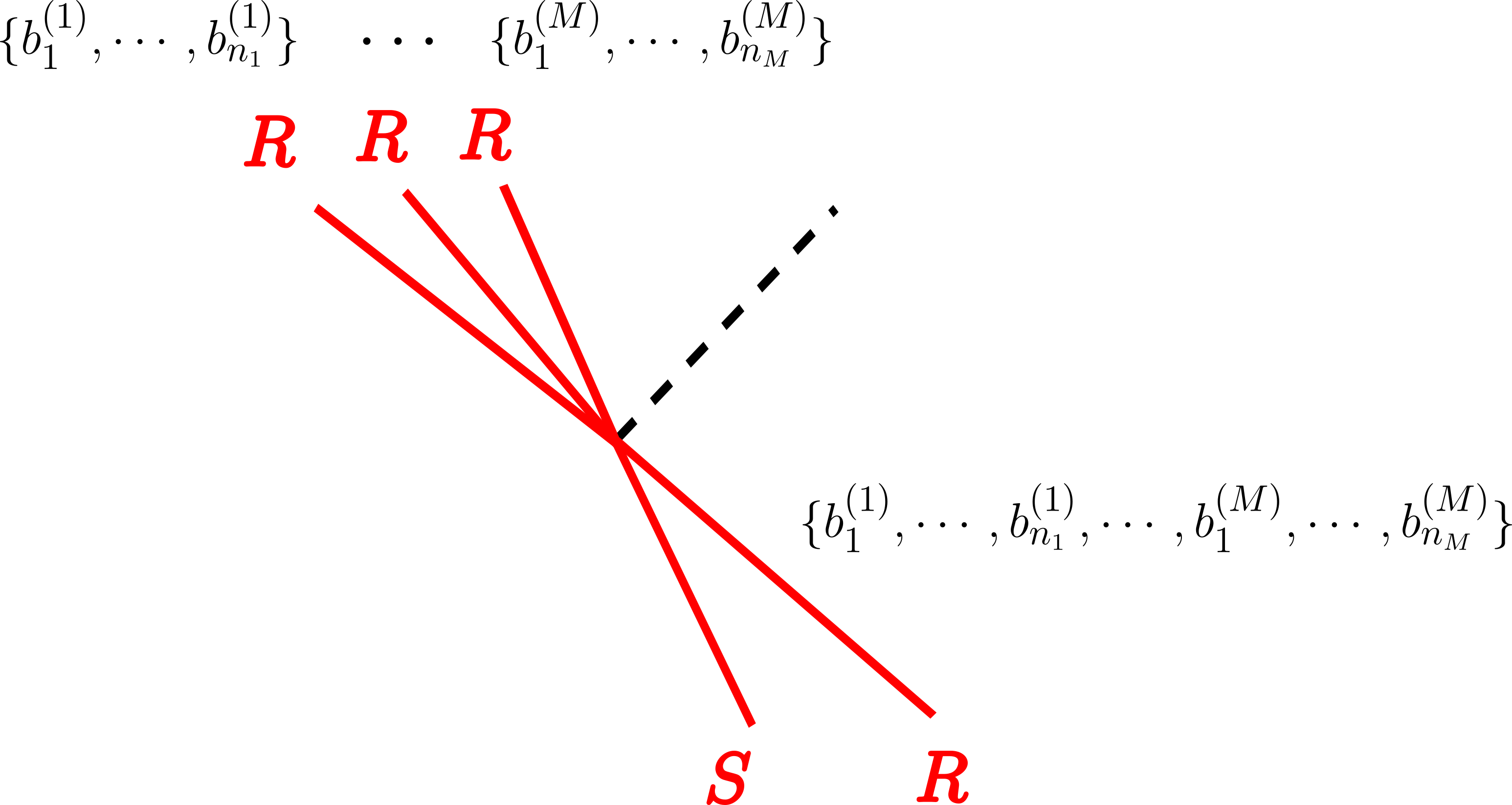}
\end{subfigure}
\hfill
\begin{subfigure}{0.45\textwidth}
    \centering
    \includegraphics[width=\linewidth]{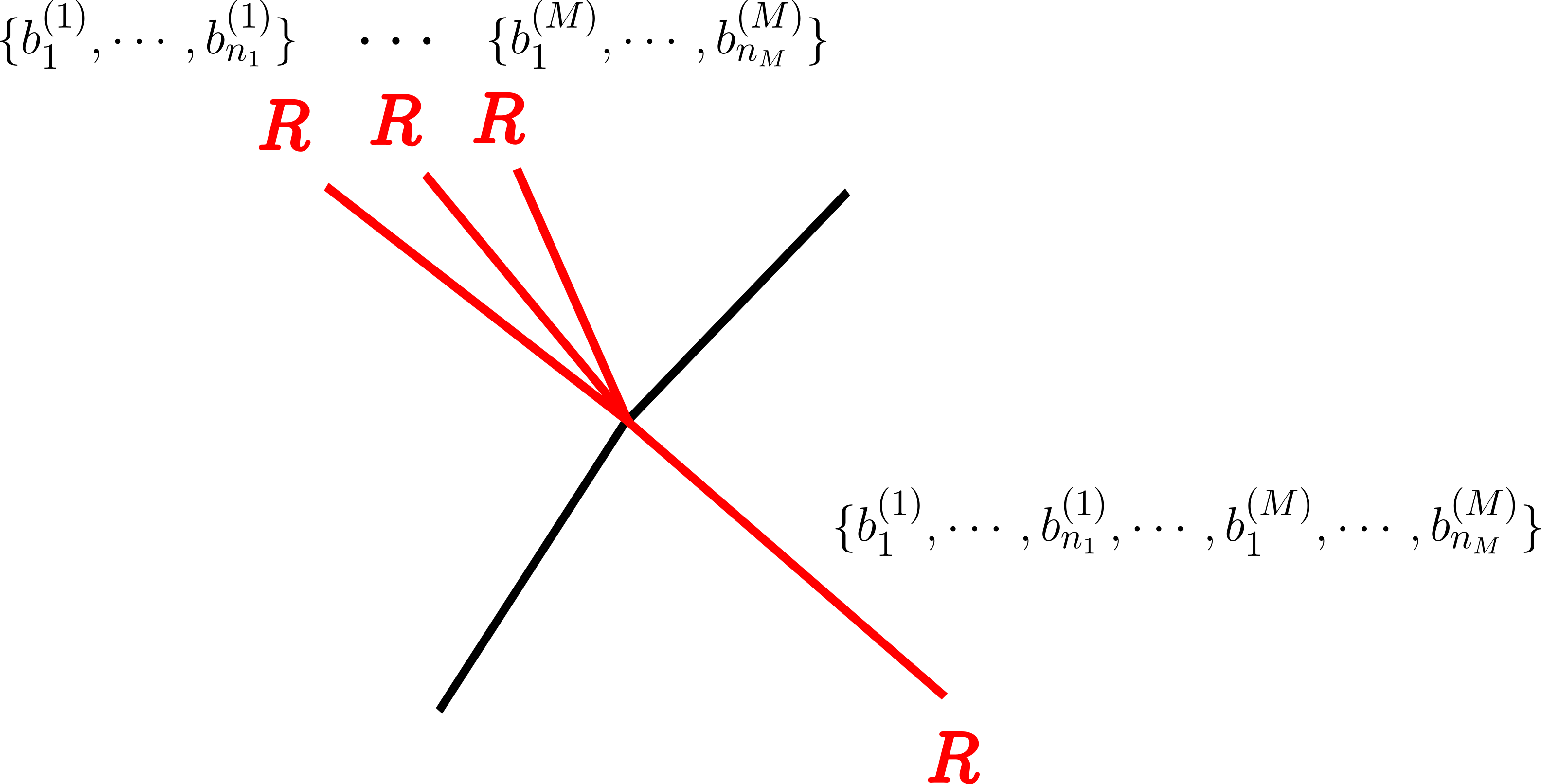}
\end{subfigure}
\caption{The outgoing waves contain exactly $M\in\mathbb{N}$ $i$-rarefactions.}
\label{Rcase1}
\end{figure}

\begin{figure}[!htbp]
\centering
\includegraphics[width=0.4\textwidth]{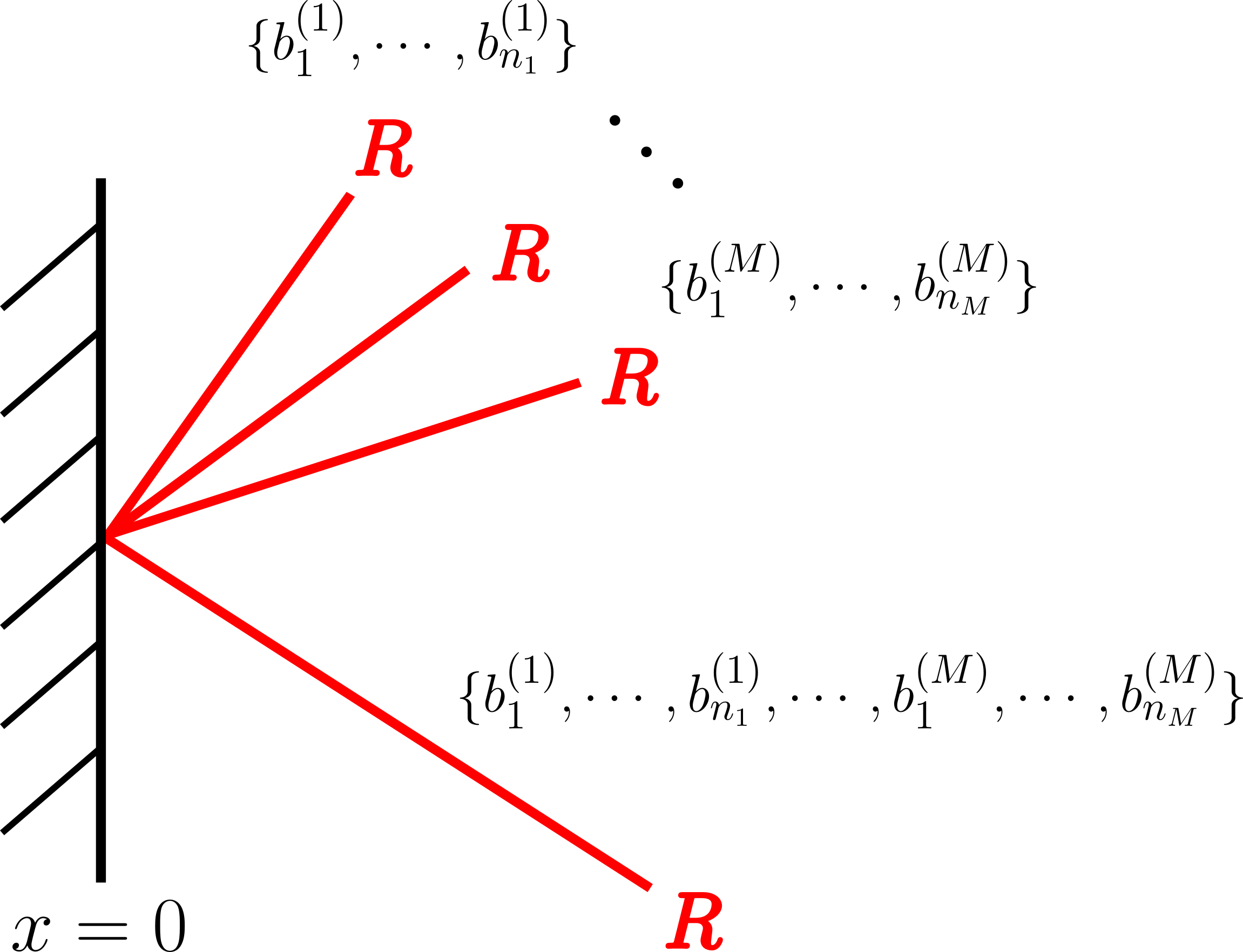}
\caption{$1$-rarefaction hits the boundary.}
\label{Rcase2}
\end{figure}

\begin{figure}[!htbp]
\centering
\includegraphics[width=0.3\textwidth]{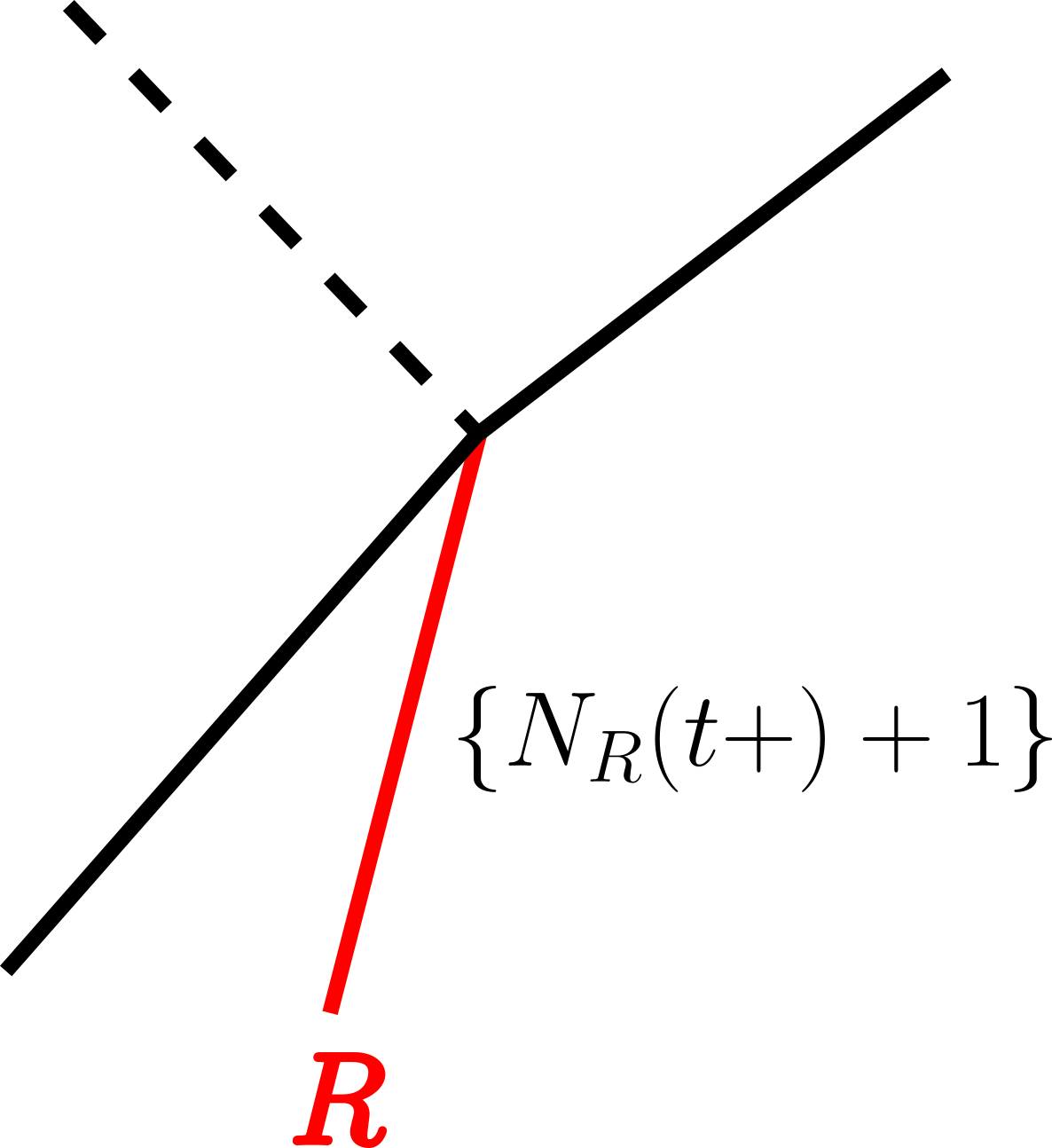}
\caption{None of the outgoing waves is an $i$-rarefaction.}
\label{Rcase3}
\end{figure}
\subsection{Evolution of wave strength}\label{subsec evolution}
By the definitions in Section~\ref{subsec label}, at any time $t$ between two consecutive events, if a wave with associated set $\{b_1,\ldots,b_n\}$ has strength $\s$, then
\begin{align*}
    \s = \sum_{l=1}^n \s_{b_l} (t) \, .
\end{align*}

We now show that the tracked strength of each element is bounded by its initial strength, up to a constant. As a first step, we state a lemma describing the change in tracked strengths at an event.

\begin{lem}\label{lem:iint}
Assume that an event occurs at time $\ts$. Let $b$ be an element belonging both to the set associated with an incoming wave and to the set associated with an outgoing wave. Then $\s_b(\ts+)$ and $\s_b(\ts-)$ have the same sign.

Assume that an interaction occurs. If the two incoming waves belong to the same family and are of different types, then
\begin{equation*}
    |\s_b(\ts +)| \leq |\s_b(\ts -)| \, .
\end{equation*}
Otherwise,
\begin{equation*}
        |\s_b(\ts +) - \s_b(\ts -)| \leq C_0 |\s_b(\ts-)| |\s| \, ,
\end{equation*}
where $C_0>0$ is the constant from Lemma~\ref{lem:Glimm}, and $\s$ denotes the strength of the incoming wave whose associated set does not contain $b$.

If a $1$-wave hits the boundary, then
\begin{align*}
    |\s_b(\ts +)| \leq \lct |\s_b(\ts -)| \, ,
\end{align*}
where $C_b>0$ is defined in \eqref{def of lct}.

\end{lem}
\begin{proof} 
Assume that an interaction occurs. We denote by $\widetilde{\s}$ and $\s'$ the strengths of the incoming and outgoing waves, respectively, whose associated sets contain $b$. We denote by $\s$ the strength of the other incoming wave.

Recall the definition of tracked strength in Section \ref{subsec label} and Lemma \ref{lem:Glimm}.
If the two incoming waves belong to the same family and are of different types, then
\begin{align*}
    |\s_b(\ts+)| - |\s_b(\ts-)| \leq \frac{|\s_b(\ts-)|}{|\widetilde{\s}|} (|\s '| - |\widetilde{\s}|) \leq 0 \, .
\end{align*}
If both incoming waves are shocks in the same family, then 
\begin{align*}
    |\s_b(\ts+) - \s_b(\ts-)| = \frac{|\s_b(\ts-)|}{|\s|+|\widetilde{\s}|} (|\s ' - (\s + \widetilde{\s})|) \leq C_0 |\s_b(\ts-)| |\s|.
\end{align*}
Otherwise, 
\begin{align*}
    |\s_b(\ts+) - \s_b(\ts -)| = \frac{|\s_b(\ts -)|}{|\widetilde{\s}|} |\s' - \widetilde{\s}| \leq C_0 |\s_b(\ts-)| |\s| \, .
\end{align*}
Assume that an incoming $1$-wave of strength $\s$ hits the boundary and generates outgoing waves with total strength $\s'$. Then
\begin{align*}
    |\s_b(\ts+)| = |\frac{\s_b(\ts-)}{\s} \s'| \leq \lct |\s_b(\ts-) | \, .
\end{align*}
\end{proof}

\begin{conv}\label{def of generation time}
For every element introduced in constructing the sets associated with shocks and rarefactions, we define its generation time to be the earliest time at which its tracked strength is defined. For an element $b$, we denote this time by $t_0^b$.
\end{conv}

Although the sets associated with rarefactions are constructed backward in time, the generation time of an element is always understood with respect to the forward time direction. 

We show that the sum of the tracked strengths of all elements at their respective generation times is controlled by the total variation at time $0$.

\begin{lem}\label{lemma total strength of all elements}
    Let $T>0$.
    We denote by $\mathcal{B}$ the set of all elements introduced throughout the evolution in constructing the sets associated with shocks and rarefactions in Section~\ref{subsec label}. 
    We have
    \begin{align*}
        \sum_{b\in\mathcal{B}} \s_b(t_0^b) \leq (1+2C_0) L(0) \, ,
    \end{align*}
    where $t_0^b$ denotes the generation time of $b$, as defined in Convention~\ref{def of generation time}, and $C_0>0$ is the constant from Lemma~\ref{lem:Glimm}
\end{lem}
\begin{proof}
We partition $\mathcal{B}$ into the set $\mathcal{B}_0$ of elements generated at time $0$ and the set $\mathcal{B}_1$ of elements with positive generation time. Then
\begin{align*}
    \sum_{b\in \mathcal{B}_0} |\s_{b}(t_0^{b})| = \sum_{b\in \mathcal{B}_0} |\s_{b}(0)| \leq L(0) \, .
\end{align*}
Every element in $\mathcal{B}_1$ is generated when two waves of the same family interact, and it belongs to the set associated with an outgoing wave of the other family. Hence, Lemma~\ref{lem:Glimm}, \eqref{nonincreasing V+Q}, and \eqref{nonincreasing V+Q interaction} imply that
\begin{align*}
    \sum_{b\in \mathcal{B}_1}|\s_{b}(t_0^{b})|
    &\leq C_0 \sum_{t>0} |\Delta Q(t)| 
    \\&\leq C_0 \sum_{t>0}|\del L(t) + \kappa \del Q(t)| 
    \\&= -C_0\sum_{t>0}\big(\del L(t) + \kappa \del Q(t)\big)
    \\&\leq C_0(L(0)+\kappa Q(0)) \, .
\end{align*}
\end{proof}

We are now ready to prove that the tracked strength of each element is bounded by its strength at the time of generation. 

\begin{lem}\label{lem:ithbd} 
    For any $U_b\in \Omega_{sub}^{in/out}$, there exists $\varepsilon_0>0$ small enough such that the following is true. 

    \noindent
    Let $V$ be an unshifted $\nu$-approximate solution constructed in Section \ref{sec:FTA}, without multiple-wave events and without two events occurring simultaneously, and with initial condition satisfying
    \begin{align*}
        \|V(\cdot,0)-U_b\|_{L^\infty(\mathbb{R}_+)}\leq \varepsilon_0 \, , \; \|V(\cdot,0)\|_{BV(\mathbb{R}_+)} \leq \varepsilon_0 \, .
    \end{align*}
    Let $T>0$.
    We assign to each wave a set and introduce elements as in Section~\ref{subsec label}. For each element $b$ and every time $t$ from its generation until it disappears, or until time $T$ if it persists up to time $T$, we have
    \begin{align*}
        |\s_b(t)| \leq 2\lct|\s_b(t_0^b)| \, ,
    \end{align*}
    where $t_0^b$ denotes the generation time of $b$, as defined in Convention~\ref{def of generation time}, and $\lct>1$ is defined in \eqref{def of lct}.
\end{lem}
\begin{proof}
Let $T_*\leq T$ be the supremum of all $t\in[0,T]$ such that, for every element $b$ with $t_0^b<t$,
\begin{align}\label{def of T*}
    |\s_b(s+)| \leq 2\lct |\s_b(t_0^b)| \, ,
\end{align}
for every $s\in[t_0^b,T_*)$. In the proof, we set $\s_b=0$ once the element $b$ disappears.

We fix an element $b$ and a time $t_0^b \leq s \leq T_*$. 
Assume first that, at its generation time, $b$ belongs to a set associated with a $2$-wave. 
Let
\begin{align*}
    t_0^b < t^1 < \cdots < t^l \leq s \, ,
\end{align*}
be the interaction times at which the set of one of the incoming waves contains $b$. For each $t_k$, we denote by $\s_k$ the strength of the other incoming wave.

By Lemma \ref{lem:iint}, we get
\begin{equation}\label{est:stpm}
\begin{aligned}
    |\s_b(s+)| &= |\s_b(t^l+)| \\
    & \leq |\s_b(t^{l-1}+)| (1 + C_0|\s_{l}|) \\
    & \leq |\s_b(t^{l-2}+)| (1 + C_0|\s_{l-1}|) (1 + C_0|\s_{l}|) \\
    &\leq ... \\
    &\leq |\s_b(t_0^b)| \prod_{m=1}^l (1 + C_0|\s_m|)  \, .
\end{aligned}
\end{equation}
Since $1+c\leq e^c$ for any $c\in \mathbb{R}$, we have
\begin{align*}
    |\s_b(s+)| &\leq |\s_b(t_0^b)| \, \textup{exp}\, ( C_0\sum_{m=1}^l |\s_m|) \, .
\end{align*}
Together with Lemma \ref{lemma property of elements}, \eqref{def of T*}, and Lemma \ref{lemma total strength of all elements}, we obtain
\begin{align*}
    \sum_{m=1}^l |\s_m| \leq 4 \lct \sum_{b'\in \mathcal{B}} |\s_{b'}(t_0^{b'})|
    \leq 4 \lct (1+2C_0) \varepsilon_0\, .
\end{align*}
Since $e^c\leq 1+2c$ for any $0\leq c\leq 1$, we can take $0<\varepsilon_0\ll 1$ so that
\begin{align}\label{estimate of track wave strength 1}
    |\s_b(s+)| \leq |\s_b(t_0^b)| \, (1+8 C_0\lct(1+2C_0) \varepsilon_0) \leq \frac{5}{4} |\s_b(t_0^b)| \, .
\end{align}

Assume now that, at its generation time, $b$ belongs to a set associated with a $1$-wave. If $b$ is not contained in a set associated with a $2$-wave at any time in $[0,T_*]$, then the preceding argument yields the same estimate. Otherwise, let $t_*$ be the time at which the $1$-wave whose associated set contains $b$ hits the boundary.

Then, together with Lemma~\ref{lem:iint}, the preceding argument yields the following estimate for every $t_*\leq s \leq T_*$,
\begin{align}\label{estimate of track wave strength 2}
    |\s_b(s+)|\leq \frac{5}{4}|\s_b(t_*+)|\leq \frac{5}{4}\lct|\s_b(t_*-)|\leq \frac{25}{16}\lct |\s_b(t_0^b)| \, .
\end{align}
By \eqref{estimate of track wave strength 1} and \eqref{estimate of track wave strength 2}, we get that for any for every $b$ with $t_0^b\leq T_*$,
\begin{align*}
    |\s_b(s+)| \leq \frac{25}{16} \lct |\s_b(t_0^b)| \, ,
\end{align*}
for every $t_0^b\leq s\leq T_*$. This contradicts the maximality of $T_*$. Hence, $T_*=T$.
\end{proof}

\subsection{Convergence of unshifted $\nu$-approximate solutions}\label{subsec limit of approximate solutions}
Finally, with the help of Lemma \ref{lem:ithbd}, we prove that any limit (up to a subsequence) of unshifted $\nu$-approximate solutions satisfies the boundary condition.

\begin{prop}\label{lem:BC}
For any $U_b \in \Omega_{sub}^{in/out}$, there exists $\varepsilon_0>0$ small enough such that the following is true. 

\noindent
Let $T>0$ be fixed, and let $V_\nu$ be a sequence of unshifted $\nu$-approximate solutions constructed in Section \ref{sec:FTA}, without multiple-wave events and without two events occurring simultaneously, and with initial condition satisfying
\begin{align*}
    \|V_\nu(\cdot,0)-U_b\|_{L^\infty(\mathbb{R}_+)}\leq \varepsilon_0 \, , \; \|V_\nu(\cdot,0)\|_{BV(\mathbb{R}_+)} \leq \varepsilon_0 \, .
\end{align*}
Then for any $R>0$, $V_\nu$ is compactly embedded in $C_x([0,R];L_t^1([0,T]))$. In particular, there exists $V\in C_x([0,R];L_t^1([0,T]))$ such that
\begin{align}\label{lem: BC eq}
    V_\nu \rightarrow V \textup{ up to a subseqence in } C_x([0,R];L_t^1([0,T])) \, .
\end{align}
Moreover, $V$ satisfies the inflow/outflow boundary condition (\ref{small BV BC}).
\end{prop}
\begin{proof}
To show (\ref{lem: BC eq}), Aubin--Lions lemma implies that it suffices to show
\begin{align}\label{aubin-lion}
    V_\nu \in L_x^\infty([0,R]; BV_t([0,T])) \, , \; \partial_x V_\nu \in L_x^\infty([0,R];\mathcal{M}_t([0,T])) \, ,
\end{align}
where $\mathcal{M}$ denotes the space of finite signed Radon measures.

Let us fix $x_*\in (0,R]$ such that $V_\nu$ has no interaction along the line $x=x_*$. A direct computation yields
\[
\begin{aligned}
    &\|V_\nu(\cdot, x)\|_{BV([0,T])} \leq C\sum |\sigma_\b(t^\beta)|, \\
    &\|\partial_x V_\nu(\cdot, x)\|_{\mathcal{M}([0,T])} \leq C\sum \frac{1}{|\lambda_\beta(t^\beta)|}|\sigma_\beta(t^\beta)|,
\end{aligned}
\]
where the sums are taken over all pairs $(\beta,t^\beta)$ such that the wave front $\beta$ crosses the line $x=x_*$ at time $t^\beta \in [0,T]$, and $\sigma_\beta(t^{\b})$ and $\lambda_\b(t^\b)$ denote, respectively, the total strength and speed of the corresponding wave front at the crossing time. Then, Lemma \ref{lemma property of elements}, Lemma \ref{lem:ithbd}, \eqref{uniform bound on speed of the wave}, and Lemma \ref{lemma total strength of all elements} imply that
\begin{align*}
    \|V_\nu(\cdot, x)\|_{BV([0,T])} \leq 4CC_b \sum_{b \in \mathcal{B}} |\s_b(t_0^b)| \leq C C_b (1+2C_0)\varepsilon_0 \, , \\
    \|\partial_x V_\nu(\cdot, x)\|_{\mathcal{M}([0,T])} \leq \frac{4CC_b}{c} \sum_{b \in \mathcal{B}} |\s_b(t_0^b)| \leq 
    \frac{C C_b (1+2C_0)}{c}\varepsilon_0 \, .
\end{align*}
Hence, we get (\ref{aubin-lion}) and thus (\ref{lem: BC eq}). Notice that $\rho_b > 0$ and the total variation is always bounded by $C\varepsilon_0$. Hence, (\ref{lem: BC eq}) implies that
\begin{align*}
    u_\nu \rightarrow u \textup{ up to a subsequence in }C_x([0,R];L_t^1([0,T])) \, ,
\end{align*}
where we write $V_\nu = (\rho_\nu, \rho_\nu u_\nu)$ and $V=(\rho, \rho u)$.
Since $V_\nu$ satisfies (\ref{small BV BC}) for any $\nu$, $V$ satisfies (\ref{small BV BC}).
\end{proof}

\section{Proof of main theorem}\label{sec: proof of them}
In this section, we introduce the definition of the weight function $a$, present the technical lemma that allows us to stop and restart the clock whenever an interaction or a boundary hit occurs, and complete the proof of Theorem \ref{thm}. The overall structure follows Sections 5 and 6 of \cite{CKF25} closely. We highlight the differences arising from the presence of the boundary and employ the techniques developed in the previous section to complete the proof.

\subsection{Construction of the weight function}
In this subsection, we define the weight function $a$, which is piecewise constant with discontinuities along every shock curve, in such a way that the assumption (\ref{prop: s weight}) of Proposition \ref{Prop:S} holds and that $a$ remains nonincreasing in time. Before defining the weight $a$ precisely, we recall an interaction estimate 
for multiple-wave interactions established in \cite{BC95}.
\begin{prop}
    Let $\psi$ be a shifted $\nu$-approximate solution. There exists a constant $K_0$ such that the following is true.

    Let there be an interaction at $(x,t)$ between incoming waves from the second family with strengths $\sigma_i''$’s and incoming waves from the first family with strengths $\sigma_i'$’s. We denote the total strength of the outgoing wave in the first family by $\sigma_1$ and in the second family by $\sigma_2$. Then
    \begin{align*}
        |\s_2-\sum \s_i ''| + |\s_1-\sum \s_i '| \leq K_0 |\Delta \tilde{Q}| \, ,
    \end{align*}
    where $\tilde{Q}$ denotes the unweighted potential, i.e.,
    \begin{align}\label{def of unweighted potential}
        \tilde{Q} = \sum_{\textup{approaching waves}} |\s_i| |\s_j| \, .
    \end{align}
\end{prop}

We define the weight function $a$ associated with a shifted $\nu$-approximate solution $\psi$ as follows
\begin{align}\label{def of a}
    a(x,t) := \textup{exp}\Big(\frac{3C_1}{4}\big(L(t)+ \frac{3\kappa}{2} Q(t) + \mu(x,t) \big)\Big)\, ,
\end{align}
where $C_1>0$ is the constant defined in Proposition \ref{Prop:S}, and 
\begin{align*}
    \mu(x,t) =  - \underset{\textup{1-shock in }(0,x)}{\sum} |\s_1| + \underset{\textup{2-shock in }(0,x)}{\sum} |\s_2| \, .
\end{align*}
Note that $\psi$ may involve multiple-wave interactions, multiple-wave hits, and simultaneous events.
We claim that the weight $a$ satisfies the following properties, similarly to Proposition 5.4 of \cite{CKF25}.

\begin{lem}
    Let $\psi$ be a shifted $\nu$-approximate solution constructed in Section \ref{sec:SFTA}. Let $a(x,t)$ be the associated weight function defined in (\ref{def of a}). Then for $\kappa > 0$ large enough, the following is true for any $\varepsilon_0>0$ small enough. 

    Assume $\|\psi(\cdot, 0)\|_{BV(\mathbb{R}_+)}= \varepsilon_0$. Then there exists a constant $C > 0$, depending only on $\kappa$, such that
    \begin{align}\label{weight around 1}
        |a(x,t)-1|\leq C \|\psi(\cdot, 0)\|_{BV(\mathbb{R}_+)} \, .
    \end{align}
	Recall $C_1$ from Proposition \ref{Prop:S}. For every time at which no wave interaction or boundary hit occurs, and for every $x$ such that a 1-shock $\s_1$ of $\psi$ is located at $x = x_1(t)$,
	\begin{equation}\label{est:1S}
			1 - 2C_1|\s_1| \leq \frac{a(t,x_1(t)+)}{a(t,x_1(t)-)} \leq 1 - \frac{1}{2}C_1|\s_1|\, .
	\end{equation}
	For every time at which no wave interaction or boundary hit occurs, and for every $x$ such that a 2-shock $\s_2$ of $\psi$ is located at $x = x_2(t)$,
	\begin{equation}\label{est:2S}
		1 + \frac{1}{2} C_1|\s_2| \leq \frac{a(t,x_2(t)+)}{a(t,x_2(t)-)} \leq 1 + 2C_1|\s_2| \, .
	\end{equation}
	For every time at which no wave interaction or boundary hit occurs, and for almost every $x>0$,
    \begin{align}\label{decary of weight}
        a(x,t+) \leq a(x,t-) \, .
    \end{align}
\end{lem}
\begin{proof}
    By (\ref{nonincreasing V+Q}), we get (\ref{weight around 1}). Following the same argument as in the proof of Proposition 5.4 of \cite{CKF25}, we can deduce (\ref{est:1S}), (\ref{est:2S}). 

    Let $t$ be an interaction time. Following the proof of Proposition 5.4 in \cite{CKF25}, (\ref{def of V, Q}) implies
    \begin{align*}
        \frac{4}{3C_1}\big(\textup{ln} (a(x,t+))-\textup{ln}(a(x,t-))\big)
        &=\Delta L (x,t)+\frac{3\kappa}{2}  \Delta Q (x,t) + (\Delta\tilde{L})_-(x,t) + 2K_0 |\Delta \tilde{Q}(x,t)|\\
        &\leq \Delta L (x,t)+ \frac{3\kappa}{2} \Delta Q (x,t) + (\Delta{L})_-(x,t) + 2K_0 |\Delta {Q}(x,t)|\, ,
    \end{align*}
    where $(f)_- := (|f| - f)/2$ denotes the negative part of $f$, $\tilde{L}$ the unweighted total variation, and $\tilde{Q}$ the unweighted potential defined in (\ref{def of unweighted potential}). 
    Then by taking $\kappa>0$ large enough, we get
    \begin{align*}
        a(x,t+)-a(x,t-)\leq 0 \, .
    \end{align*}
    
    Let $t$ be a time when a set of 1-waves $\mathcal{S}_1$ hits the boundary and generates a set of 2-waves $\mathcal{S}_2$. We see that for almost every $x>0$, we can take $\varepsilon_0$ small enough such that (\ref{def of lct}) and (\ref{relation K_1, C_1}) give
    \begin{align*}
        &\frac{4}{3C_1}\big(\textup{ln} (a(x,t+))-\textup{ln}(a(x,t-))\big) \\
        &=\Delta (L+\frac{3\kappa}{2} Q) (x,t) + \Delta \mu(x,t)  \\
        &\leq - K_1 \sum_{\mathcal{S}_1} |\s_1| + \sum_{\mathcal{S}_2} |\s_2| + C \varepsilon_0 \kappa  K_1 \sum_{\mathcal{S}_2} |\s_2| + \sum_{\mathcal{S}_1} |\s_1| + \sum_{\mathcal{S}_2} |\s_2|\\
        &\leq -K_1( 1 - 2 \frac{\lct}{K_1} - C \varepsilon_0 \kappa  \frac{\lct}{K_1} - \frac{1}{K_1}) \sum_{\mathcal{S}_1} |\s_1| \, , \\
        &\leq -\frac{K_1}{8} \sum_{\mathcal{S}_1} |\s_1| \, .
    \end{align*}
\end{proof}


\subsection{Stopping and restarting the clock}
When we run the front tracking algorithm, it is necessary to stop and restart the clock whenever two waves interact or a wave hits the boundary. 
For the class of low-regularity solutions considered here, 
restarting the clock at positive times $t>0$ requires some care. To address this issue, we use the notion of approximate limits as follows (for background on approximate limits, see \cite{EG}). 

\begin{lem}\label{technical lemma clock} (Technical lemma for the approximate limit)\\
Let $U \in \mathcal{S}_{weak}$ with initial data $U^0$. Then for all $V \in \mathcal{V}$, and for all $c, d \in \RR$ with $0\leq c<d$, the approximate right- and left-hand limits 
\[
\operatorname*{ap} \lim_{t \to t_0 \pm} \int_c^d \eta(U(x,t)|V)\,dx
\]
exist for all $t_0 \in (0,\infty)$, and 
\[
\operatorname*{ap} \lim_{t \to t_0 -} \int_c^d \eta(U(x,t)|V)\,dx \geq \operatorname*{ap} \lim_{t \to t_0 +} \int_c^d \eta(U(x,t)|V)\,dx.
\]
Moreover, the approximate right-hand limit exists at $t_0 = 0$, and 
\[
\int_c^d \eta(U^0(x)|V)\,dx \geq \operatorname*{ap} \lim_{t \to t_0+}\int_c^d \eta(U(x,t)|V)\,dx.
\]
\end{lem}
\begin{proof}
The proof of Lemma \ref{technical lemma clock} is almost identical to that of Lemma 2.5 of \cite{clock}, with the only modification being the replacement $\mathbb{R}$ by the half-line $\mathbb{R}_+$. We therefore omit the proof.
\end{proof}

\subsection{Proof of Theorem \ref{thm}}\label{subsec proof of them main theorem}
The proof is similar to that of Section 6.2 of \cite{CKF25}, with a few modifications. 

First, the cone of information is different. In our setting, the leftmost curve $h_0(t)$ coincides with the boundary, so that $h_0'(t)=0$, whereas in \cite{CKF25} the cone is determined by a right-moving line with $h_0'(t)>0$ and $|h_0'(t)|$ sufficiently large.

In \cite{CKF25}, the strong inclination of the cone guarantees contraction. However, this mechanism is no longer available in the presence of a boundary. Instead, we extract additional control from the boundary set in Definition~\ref{def:BVP} through the estimates established in Lemma~\ref{lem: ouflow boundary} and Lemma~\ref{lem: inflow boundary}, which compensate for the fact that $h_0'(t)=0$.

Second, we need to ensure that the limit (up to a subsequence) of unshifted $\nu$-approximate solutions satisfies the inflow/outflow boundary condition \eqref{small BV BC}, a property not required in the Cauchy problem. This is achieved by applying Proposition \ref{lem:BC}. 

For completeness, we include the proof below. \vspace{3mm}

From \cite[Lemma 7.2]{qualitative}, there exists $s_0>0$ (depending on $\mathcal{V}_0$ and $\mathcal{O}$) such that for any $W_1\in \mathcal{V}_0$ and any $W_2 \in \overline{\mathcal{O}}$, 
\begin{align}\label{q controlled by eta}
    |q(W_1;W_2)|\leq s_0 \eta(W_1|W_2) \, .
\end{align}
Let $\Lambda$ denote the largest absolute value of the speed of any wave appearing in either the shifted or unshifted $\nu$-approximate solutions. We define
\begin{align}\label{def of s}
    s := \max\{s_0, \Lambda\} \, .
\end{align}

\noindent
\textit{Step 1: the dissipation calculations.}
For a fixed $T>0$, we take a sequence of unshifted $\nu$-approximate solutions $V_\nu$ such that
\begin{itemize}
\item[1)] their initial data $V_\nu(\cdot,0)$ are step functions that converge to $V^0(\cdot,0)$ in $L^2(\mathbb{R}_+)$ as $\nu \to 0$.
\item[2)] they contain no multiple-wave interactions and no multiple-wave hits, with at most one interaction or hit occurring at any time in $[0,T]$.
\end{itemize}
Such a sequence exists, because the number of interactions and wave fronts are finite on any compact time interval and only depend on the total variation of the initial value and $\nu$. The method developed in the proof of Lemma 9 in \cite{BC95} can provide an explicit bound. Hence, for each $\nu$, we can obtain such a $V_\nu$ by finitely many small perturbations on the initial data.

We fix $\nu>0$ and let $\psi$ be the shifted $\nu$-approximate solution constructed in Section \ref{sec:SFTA} with initial condition $\psi(\cdot,0)=V_\nu(\cdot,0)$.
Let $0<\tau < T$, and let $t_0=0$ and $t_j<\tau<t_{j+1}$ be two successive times at which either an interaction or a hitting happens. 

For a fixed $R>0$, we restrict our analysis to the cone of information
\begin{align}\label{def of h_0, h_N+1}
    h_0(t) = 0 \, , \; \mathbf{h}(t) := h_{N+1}(t) = R-s(t-\tau) \, ,
\end{align}
and denote the discontinuity curves between $t_j$ and ${t_{j+1}}$ by $h_1,...,h_N$ for some $N\in \mathbb{N}$ such that on $(t_j,t_{j+1})$,
\begin{align*}
    h_0< h_1 < h_2 < \cdots < h_N < h_{N+1}  \, .
\end{align*}
Since $s$ exceeds the speed of every front, no wave enters the cone. Without loss of generality, we also assume that $h_1,\dots,h_N$ remain within the cone; otherwise, we may stop and relabel accordingly.

Let $t\in[t_j,t_{j+1}]$. We integrate (\ref{entropy inequality}) on
\begin{align*}
    Q = \{(x,r)\, |\, h_i(r) < x < h_{i+1}(r) \, , \; t_j<r<t\} \, ,
\end{align*}
where $\psi$ remains a constant, and obtain
\begin{align*}
    &\operatorname*{ap} \lim_{r \to t-} \int_{h_i(t)}^{h_{i+1}(t)} a(x,t-)\eta(U(x,r)|\psi(x,t)) \, dx\\
	&\quad \leq \operatorname*{ap} \lim_{r \to t_j+} \int_{h_i(t_j)}^{h_{i+1}(t_j)} a(x,t_j+)\eta(U(x,r)|\psi(x,t_j)) \, dx + \int_{t_j}^t F_i^+(r) - F_{i+1}^-(r)\,dr \,,
\end{align*}
where
\[
\begin{aligned}
	&F_i^+(r) = a(h_i(r)+,r)\Big[q(U(h_i(r)+,r);\psi(h_i(r)+,r)) - \dot{h}_i(r) \eta(U(h_i(r)+,r)|\psi(h_i(r)+,r)) \Big],\\
	&F_i^-(r) = a(h_i(r)-,r)\Big[q(U(h_i(r)-,r);\psi(h_i(r)-,r)) - \dot{h}_i(r) \eta(U(h_i(r)-,r)|\psi(h_i(r)-,r)) \Big].
\end{aligned}
\]
We sum over $i$ from $0$ to $N$ and get
\[
\begin{aligned}
	&\operatorname*{ap} \lim_{r \to t-} \int_{0}^{R-s(t-\tau)} a(x,t-)\eta(U(x,r)|\psi(x,t))\,dx\\
	&\quad \leq \operatorname*{ap} \lim_{r \to t_j+} \int_{0}^{R-s(t_j-\tau)} a(x,t_j+)\eta(U(x,r)|\psi(x,t_j))\,dx \\
	&\qquad \quad + \int_{t_j}^t F_0^+(r) - F_{N+1}^-(r)\,dr +  \sum_{i=1}^N \int_{t_j}^t F_i^+(r) - F_{i}^-(r)\,dr \, .
\end{aligned}
\]
Since (\ref{def of h_0, h_N+1}) and (\ref{q controlled by eta}), we get $F_{N+1}^- \geq 0$. Lemma \ref{lem: ouflow boundary}, Lemma \ref{lem: inflow boundary}, and $h_0(t) = 0$ imply
\[
F_0^+(r) = a(0+,r)q(u(0+,r);\psi(0+,r)) \leq 0 \, .
\]
By Proposition \ref{Prop:S} and Proposition \ref{Prop:R}, we get
\begin{align*}
    &\operatorname*{ap} \lim_{r \to t-} \int_{0}^{R-s(t-\tau)} a(x,t-)\eta(U(x,r)|\psi(x,t))\,dx\\
	&\leq \operatorname*{ap} \lim_{r \to t_j+} \int_{0}^{R-s(t_j-\tau)} a(x,t_j+)\eta(U(x,r)|\psi(x,t_j))\,dx + C\nu (t-t_j) \\
	&\quad -C\int_{t_j}^t \underset{{\alpha\in \mathcal{S}_\tau(r)}}{\sum} |\psi(h_\alpha(r)+,r)-\psi(h_\alpha(r)-,r)| |\lambda_{\alpha_i}^\phi(\psi(h_\alpha(r)-,r),\psi(h_\alpha(r)+,r))-\dot{h}_\alpha(r)|^2  dr ,
\end{align*}
where $\mathcal{S}_{\tau}(r)$ denotes the set of shocks inside the cone of information at time $r$, that is, the shocks located in $[0,\mathbf{h}(r)] = [0,R-s(r-\tau)]$.

Now, we sum over all the interacting and hitting times after $\tau$. The convexity of $\eta$, Lemma \ref{technical lemma clock}, and the non-increasing property of the weight (\ref{decary of weight}) give (see Section 6.2 of \cite{CKF25} for details)
\begin{align*}
        &\int_{0}^R a(x,\tau) \, \eta(U(x,r)|\psi(x,\tau))\, dx - \int_{0}^{R+s\tau} a(x,0) \, \eta(U(x,0)|\psi(x,0)) \, dx \\
        &\leq -C\int_0^\tau \underset{{\alpha\in \mathcal{S}_\tau(r)}}{\sum} |\psi(h_\alpha(r)+,r)-\psi(h_\alpha(r)-,r)| |\lambda_{\alpha_i}^\phi(\psi(h_\alpha(r)-,r),\psi(h_\alpha(r)+,r))-\dot{h}_\alpha(r)|^2 dr \\
        &\quad + C\nu \, ,
\end{align*}
which implies
\begin{align}\label{main thm proof eq 1}
    \int_{0}^R a(x,\tau) \, \eta(U(x,r)|\psi(x,\tau))\, dx
    \leq \int_{0}^{R+s\tau} a(x,0) \, \eta(U(x,0)|\psi(x,0)) \, dx + C\nu \, ,
\end{align}
and
\begin{equation}\label{main thm proof eq 3}
    \begin{aligned}
        \int_0^\tau \underset{{\alpha\in \mathcal{S}_\tau(r)}}{\sum} |\psi(h_\alpha(r)+,r)-\psi(h_\alpha(r)-,r)| |\lambda_{\alpha_i}^\phi(\psi(h_\alpha(r)-,r),\psi(h_\alpha(r)+,r))-\dot{h}_\alpha(r)|^2 dr \\
        \leq \int_{0}^{R+s\tau} a(x,0) \, \eta(U(x,0)|\psi(x,0)) \, dx + C\nu \, .
    \end{aligned}
\end{equation}
Note that $C\nu$ bounds two types of error:
\begin{itemize}
    \item[1)] The discrepancy between the shock curve and the interpolated shock curve (\ref{interpolated curve}), arising when applying Proposition~\ref{Prop:S};
    \item[2)] The accumulation over time of the right-hand side of (\ref{prop R ineq}), in the context of Proposition~\ref{Prop:R}.
\end{itemize}
\textit{Step 2:} Again, we follow the proof shown in Section 6.2.2 of \cite{CKF25}. Proposition \ref{prop L^1 semigroup}, Lemma \ref{lem: L1 difference}, and (\ref{main thm proof eq 1}) imply that 
\begin{align}\label{main thm proof eq 2}
    \|V_\nu(\cdot, \tau)-\psi(\cdot,\tau)\|_{L^1(0,R)} \leq C \big(\int_0^{R+s\tau} a(x,0)\, \eta(U(x,0)|\psi(x,0)) \, dx\big)^{1/2} + C \nu^{1/2} \, .
\end{align}
There is a small subtlety in restricting the estimate in Lemma~\ref{lem: L1 difference} to the cone of information, which is implicit in \cite{CKF25}. Nevertheless, the argument carries over verbatim here and is unaffected by the presence of the boundary. We provide the details below.

Although Lemma~\ref{lem: L1 difference} is stated on the whole domain $\mathbb{R}_+$, the estimate (\ref{main thm proof eq 2}) can be localized to the cone of information. We define a modified $\nu$-approximate solution $\tilde{\psi}$ by propagating fronts inside the cone according to the shifted construction as in Section~\ref{sec:SFTA}, and propagating fronts outside the cone according to the unshifted construction as in Section~\ref{sec:FTA}. Applying Lemma~\ref{lem: L1 difference} to $\tilde{\psi}$ and $V_\nu$, only shocks inside the cone contribute to the estimate. Since the cone speed $s$ is larger than all wave speeds, $\tilde{\psi}$ agrees with $\psi$ inside the cone, and~\eqref{main thm proof eq 2} follows.

By Holder's inequality, (\ref{main thm proof eq 1}), and (\ref{main thm proof eq 2}), we get
\begin{align*}
    \|U(\cdot,\tau)-V_\nu(\cdot,\tau)\|_{L^1(0,R)} 
    &\leq \|U(\cdot,\tau)-\psi(\cdot,\tau)\|_{L^1(0,R)} + \|\psi(\cdot,\tau)-V_\nu(\cdot,\tau)\|_{L^1(0,R)} \\ 
    &\leq \sqrt{R} \, \|U(\cdot,\tau)-\psi(\cdot,\tau)\|_{L^2(0,R)} + \|\psi(\cdot,\tau)-V_\nu(\cdot,\tau)\|_{L^1(0,R)} \\
    &\leq C \|U(\cdot,0)-V_\nu(\cdot,0)\|_{L^2(0,R)} + C \nu^{1/2} \, .
\end{align*}
As the state space $\mathcal{V}_0$ is bounded, we get
\begin{align*}
    \|U(\cdot,\tau)-V_\nu(\cdot,\tau)\|_{L^2(0,R)} \leq C \|U(\cdot,0)-V_\nu(\cdot,0)\|_{L^2(0,R)}^{1/2} + C \nu^{1/4} \, .
\end{align*}
By letting $\nu \to 0$, Proposition~\ref{lem:BC} implies that, up to a subsequence, $V_\nu$ converges to a function $V$ in $C_{loc}(\mathbb{R}_+;L^1[0,T])$, in particular, $L^1_{loc}((0,T) \times \mathbb{R}_+)$. In addition, $V$ satisfies the inflow/outflow boundary condition \eqref{small BV BC}. By the front-tracking construction and the convergence theory of \cite{AC97}, the limit $V$ is an entropic weak solution of \eqref{eq:E}. 
Moreover, the uniform BV and $L^\infty$ bounds together with the Lipschitz continuity in time of the approximate solutions allow us to apply Theorem~2.4 in \cite{l1stability}, which yields that $V$ has small bounded variation for all times. 
Passing to the limit in the stability estimate then gives
\begin{align*}
    \|U(\cdot,\tau)-V(\cdot,\tau)\|_{L^2(0,R)} \leq C \|U(\cdot,0)-V(\cdot,0)\|_{L^2(0,R)}^{1/2} \, .
\end{align*}
This completes the proof of Theorem~\ref{thm}.

\vspace{0.3cm}
\noindent \textbf{Statements and Declarations}\\

\noindent \textbf{Data availability:} No datasets were generated or analyzed during the current study. \vspace{0.3cm}

\noindent \textbf{Conflict of interest:} The authors declare that they have no conflicts of interest with this work.

\bibliography{ref.bib}
\end{document}